%% file: 0main.tex
\documentclass[reqno]{amsart}
\usepackage{comment}
\usepackage{amsthm}
\usepackage{amssymb}
\usepackage{amsmath}
\usepackage{bbm}
\usepackage{tikz-cd}
\RequirePackage{tikz-cd}
\RequirePackage{amssymb}
\usetikzlibrary{calc}
\usetikzlibrary{decorations.pathmorphing}

\tikzset{curve/.style={settings={#1},to path={(\tikztostart)
    .. controls ($(\tikztostart)!\pv{pos}!(\tikztotarget)!\pv{height}!270:(\tikztotarget)$)
    and ($(\tikztostart)!1-\pv{pos}!(\tikztotarget)!\pv{height}!270:(\tikztotarget)$)
    .. (\tikztotarget)\tikztonodes}},
    settings/.code={\tikzset{quiver/.cd,#1}
        \def\pv##1{\pgfkeysvalueof{/tikz/quiver/##1}}},
    quiver/.cd,pos/.initial=0.35,height/.initial=0}

\tikzset{tail reversed/.code={\pgfsetarrowsstart{tikzcd to}}}
\tikzset{2tail/.code={\pgfsetarrowsstart{Implies[reversed]}}}
\tikzset{2tail reversed/.code={\pgfsetarrowsstart{Implies}}}
\tikzset{no body/.style={/tikz/dash pattern=on 0 off 1mm}}
\usepackage[left=3.3cm,right=3.3cm,bottom=3.7cm,asymmetric]{geometry}
\usepackage{mlmodern}
\usepackage[T1]{fontenc}

\usepackage[backend=bibtex, 
bibstyle=ext-numeric, 
citestyle= ext-numeric, 
giveninits = true,
sorting=nyt,
doi=false,
isbn=false,
maxnames=50]{biblatex} 
\usepackage{biblatex-shortfields}
\renewbibmacro{in:}{}

\DeclareFieldFormat{pages}{#1}

\DeclareFieldFormat{labelnumber}{\printtext{\upshape \bfseries #1}}
\DeclareFieldFormat{labelnumberwidth}{\mkbibbold{#1\adddot}}

\renewbibmacro*{volume+number+eid}{%
  \printfield{volume}%
  \printfield{number}%
  \setunit{\bibeidpunct}%
  \printfield{eid}}

\DeclareFieldFormat[article]{number}{ \mkbibparens{#1}}

\DeclareFieldFormat[article]{issuedate}{\printtext[parens]{#1}\printunit{\addspace}}

\DeclareFieldFormat[inbook]{series}{#1\nopunct}
\DeclareFieldFormat[inbook]{number}{#1\nopunct}
\DeclareFieldFormat[book]{number}{#1\nopunct}
\DeclareFieldFormat[book]{series}{#1\nopunct}
\DeclareFieldFormat[incollection]{booktitle}{\itshape#1\nopunct}

\renewbibmacro*{publisher+location+date}{%
  \printtext[parens]{
    \printlist{publisher}%
    \setunit*{\addcomma\space}%
    \printlist{location}%
    \setunit*{\addcomma\space}%
    \usebibmacro{date}%
  }\printunit{\addspace}}

\DeclareFieldFormat[misc,article,thesis]{title}{`#1'}
\DeclareFieldFormat{url}{\url{#1}}

\renewbibmacro*{byeditor+others}{%
  \ifnameundef{editor}
    {}
    {\printtext[parens]{\usebibmacro{editor+othersstrg}%
     \setunit{\addspace}%
     \printnames[byeditor]{editor}}%
     \clearname{editor}%
     \newunit}%
  \usebibmacro{byeditorx}%
  \usebibmacro{bytranslator+others}}

\usepackage{xcolor}
\usepackage{hyperref}
\hypersetup{
	colorlinks=true,
	linkcolor=purple,
	citecolor=teal,
	linktoc=page
}
\usepackage{fancyhdr}
\theoremstyle{definition}
\newtheorem{defin}{Definition}[section]

\theoremstyle{remark}
\newtheorem{example}[defin]{Example}
\newtheorem{remark}[defin]{Remark}

\theoremstyle{plain}
\newtheorem{lem}[defin]{Lemma}
\newtheorem{prop}[defin]{Proposition}
\newtheorem{thm}[defin]{Theorem}
\newtheorem{cor}[defin]{Corollary}
\newtheorem*{claim}{Claim}

\newtheorem{mainthm}{Theorem}

\newcommand{\C}{\mathcal{C}}
\newcommand{\M}{\mathcal{M}}
\newcommand{\N}{\mathcal{N}}
\newcommand{\E}{\mathcal{E}}
\newcommand{\U}{\mathcal{U}}
\newcommand{\Lang}{\mathcal{L}}
\newcommand{\Hom}{\textup{Hom}}
\newcommand{\Fun}{\textup{Fun}}
\newcommand{\Pro}{\textup{Pro}}
\newcommand{\defT}{\textup{def}}
\newcommand{\tpdefT}{\textup{tpdef}}
\newcommand{\tp}{\textup{tp}}

\DeclareMathOperator*{\colim}{colim}

\title{Hyperimaginaries and Exactness of the Pro-Completion}
\author{Owen Ngo Hang Chan}
\address{Department of Mathematics, University of Manchester, Oxford Road, Manchester, United Kingdom M13 9PL}
\email{owen.chan@manchester.ac.uk}
\urladdr{https://sites.google.com/view/ngo-hang-chan/}
\date{\today}
\subjclass[2020]{03C45, 03C95, 18E08}
\keywords{hyperimaginaries, pro-completion, exact categories}

\begin{document}

\begin{abstract}
    Given a complete first-order theory $T$, we characterise elimination of hyperimaginaries in $T$ in terms of the exactness of the pro-completion $\Pro(\defT(T))$ of the syntactic category $\defT(T)$ of $T$, thus extending Makkai's result connecting elimination of imaginaries with the exactness of $\defT(T)$. Likewise, we characterise the heq construction in terms of the exact completion of $\Pro(\defT(T))$. 
\end{abstract}

\maketitle

\tableofcontents
\input{1Introduction}

\input{2eq_and_exact.tex}

\input{3Basic_properties_of_pro.tex}

\input{4elimination_of_hyperimaginaries.tex}

\input{5heq.tex}

\bigskip

\linespread{1}
\printbibliography
\end{document}

%% file: 1Introduction.tex
\section{Introduction}

While establishing classification theory, Shelah introduced the eq construction of a first-order theory $T$ and studied what he called imaginary elements \cite[Ch.~III \S6]{shelahbook}. This was done so that every set definable with parameters has a minimal set of defining parameters, called its canonical parameter. The idea is simple: an imaginary element is an equivalence class of a definable equivalence relation; the eq construction expands a theory $T$ to $T^{eq}$ by freely adding in quotients of definable equivalence relations as new sorts. In \cite{poizat_imaginary_galois}, Poizat defined that $T$ eliminates imaginaries if such quotients can be naturally identified with a definable set in $T$. He went on to show that the theory of algebraically closed fields and the theory of differentially closed fields with characteristic zero both eliminate imaginaries, thereby establishing a model-theoretic generalisation of classical Galois theory and differential Galois theory. More recently, Hrushovski \cite{hrushovski_groupoid} studied generalised imaginaries, replacing definable equivalence relations with definable groupoids, which are the vertical categorification of equivalence relations. He showed that generalised imaginaries are related to internal covers of $T$ and type amalgamation. This is subsequently generalised in \cite{GoodrickJohn2013Afab,GoodrickJohn2014Tpap,GoodrickJohn2010Gca3,kamensky_cat_app_int,kamensky_higher_int_cov,moosa_haykazyan,wang_simp_grpd}, providing a fecund field for interactions between model theory and category theory, and indeed higher category theory.

\smallskip

On the other hand, category theorists were also studying equivalence relations internal to a category and the existence of their quotients, as well as the free completion of a category under these quotients. A category with all such quotients is called exact\footnote{There are two unrelated notions of exactness in category theory, called Barr-exactness and Quillen-exactness. The exactness considered in this paper is Barr-exactness.}; the free construction is called the ex/reg completion; we write $\C_{ex/reg}$ for the ex/reg completion of a regular category $\C$. The subscript `ex' stands for exact and `reg' stands for regular. The ex/reg completion freely turns a regular category into an exact category. This is to distinguish it from the other exact completion---the ex/lex completion---which freely turns a lex category---a category with finite limits---into an exact category. The first references come from Barr \cite{Barrbook} and Lawvere \cite{Lawvere_Perugia}. 
It is evident that Lawvere had already known about what was eventually named ex/reg completion. Indeed, in \cite[pp.~84--85]{Lawvere_Perugia}) he sketched out the construction, but it would be Succi Cruciani \cite{cruciani}, and Carboni and Vitale \cite{CarboniA.1998Raec}, who proved that the construction described by Lawvere is indeed the free construction and has the expected universal property. Lack \cite{LackStephen1999Anot} provides another description for this construction in terms of sheaves, which was later generalised by Shulman \cite{ShulmanMichael2012Ecas}. There is a related construction---the pretopos completion of a coherent category---which often coincides with the ex/reg completion; for example, see \cite[Proposition 5.1]{EmmeneggerJacopo2020Edac}. As a result, there is potential for confusion for which construction is relevant in which setting. Shulman's work \cite{ShulmanMichael2012Ecas} provides a framework of seeing the ex/reg completion as the finitary version of the pretopos completion.

\smallskip

The connection between the model-theoretic approach and the category-theoretic approach is attributed to Makkai, who, according to Harnik's expository paper \cite{harnik}, knew already in 1980 that `$T^{eq}$ is nothing but the pretopos completion'. Indeed, the subject of categorical model theory was established in Makkai's joint work with Reyes in \cite{alma992977584040101631}. Briefly, given a complete first-order theory $T$, one can define the syntactic category of $T$ as the category of definable sets and definable functions, denoted by $\defT(T)$. Since our logic is first-order, it is in particular regular, which is the fragment of first-order logic consisting of truth $\top$, finite conjunction $\wedge$, equality $=$, and existential quantification $\exists$. Hence, $\defT(T)$ is a regular category. In fact, the inclusion of finite disjunction $\vee$ in the logic makes $\defT(T)$ coherent, and the inclusion of negation $\neg$ makes $\defT(T)$ Boolean. The syntactic category encodes all the information of $T$ in a language-agnostic way: the category does not `see' whether a formula is quantifier-free or not. Nonetheless, much model theory can be done by studying $\defT(T)$; for example, the category $\textup{Coh}(\defT(T),\textbf{Set})$ of coherent functors into the category of sets is equivalent to the category of models and elementary embeddings \cite[Ch. 8 \S1]{alma992977584040101631}. 

\smallskip 

The work of Makkai and Reyes includes the celebrated result known as conceptual completeness \cite[Ch. 7]{alma992977584040101631}. This says that if an interpretation of a theory $T_1$ in another theory $T_2$ induces an equivalence between the two categories of models, then this interpretation is in fact a bi-interpretation up to the eq construction, or equivalently, up to the pretopos completion. This is perhaps another reason why the eq construction has become synonymous with the pretopos completion, as seen in \cite{harnik}. In \S\ref{preliminaries}, we will review this and will find that, it is exactness, rather than the property of being a pretopos, that is relevant to imaginaries. Indeed, we see in Theorem~\ref{degenerate theorem} that certain non-degenerate assumptions are needed for the pretopos completion of $\defT(T)$ to coincide with $\defT(T^{eq})$. On the contrary, as we shall see in Theorem \ref{Teq is exact}, no such non-degenerate assumptions are needed to prove 
\[\defT(T^{eq})\simeq\defT(T)_{ex/reg}.\]

In \S\ref{pro}, we study the pro-completion $\Pro(\defT(T))$ of $\defT(T)$ with a view of generalising in \S\ref{hyperimaginaries} the connection between imaginaries and exactness to a connection between the projective limits of imaginaries, called hyperimaginaries, and the exactness of $\Pro(\defT(T))$. While instances of projective limits, or equivalently cofiltered limits, have been studied for over a century---Hensel was studying the $p$-adic numbers as early as 1897---the completion $\Pro (\C)$ of a category $\C$ under projective limits was first introduced by Grothendieck \cite{Grothendieck1958-1960}. Similar to the ex/reg completion, the pro-completion freely adds projective limits into a category. This completion can be described as a certain subcategory of the category of copresheaves on $\C$. When one dualises this construction, one obtains the ind-completion, which freely adds in direct limits, or equivalently filtered colimits \cite[Expos\'{e} I \S8.10]{a.TheorieToposCohomologie1972}. 

\smallskip

Barr \cite{BARR1986113} studied $\Pro (\C)$ in the case when $\C$ is regular. It turns out that many nice properties of $\C$ transfer to $\Pro(\C)$, such as regularity and coherence. Some properties do not transfer, for example, being Boolean. Model-theoretically, this is saying that the complement of a type-definable set is a $\bigvee$-type definable set, that is, definable with infinitary disjunction, rather than a type-definable set. Another property that does not transfer is exactness. 

\begin{example}\label{finset}
    Let $\C$ be the category \textbf{FinSet} of finite sets and all functions, which is a pretopos and in particular exact. Its pro-completion is the category \textbf{Stone} of Stone spaces and continuous functions (see \cite[Ch.~VI \S 2.3]{alma9912295854401631}), which is not exact. Indeed, its ex/reg completion is \textbf{CHaus}, the category of compact Hausdorff spaces and continuous functions; see \cite[Theorem 8.1]{MarraVincenzo2020Acot}. This fits into the $\defT(T)$ framework: consider $T$ to be the theory of the structure with sorts $S_n$ for each $n\in\omega$, where each sort $S_n$ is interpreted as a set containing exactly $n$ elements, with every element named by a constant. Then, $\defT(T) \simeq \textbf{FinSet}$.
\end{example}

Of course, there are examples where $\Pro(\C)$ is exact. When $\C$ is the category \textbf{FinGrp} of finite groups, the pro-completion is the category \textbf{ProfGrp} of profinite groups and continuous group homomorphisms. Now, equivalence relations of profinite groups correspond to closed subgroups. Since the quotient of a profinite group by a closed subgroup is again a profinite group \cite[Proposition~2.2.1~(a)]{profinite_groups}, \textbf{ProfGrp} is exact.

\smallskip

It was Kamensky \cite{KAMENSKY2007180}, who first studied $\Pro(\defT(T))$ in the context of model theory, relating it to the types of $T$. By types, we mean $\ast$-types without parameters, which are sets of formulas in an infinite tuple of variables; likewise, given a model $\M$, a type-definable set is a subset of $\M^I$, whose elements are the realisations of a type and where $I$ is an infinite indexing set. Then, the objects of $\Pro(\defT(T))$ can be thought of as the types of $T$. The morphisms correspond to types, whose interpretation in every model is the graph of a function. Equivalently, one can fix a monster model $\U$ and ask for the interpretation in $\U$ to be the graph of a function. In \S\ref{subsect_saturation}, we show a new characterisation of saturation in terms of the preservation of certain regular epimorphisms in $\Pro(\defT(T))$; namely, a model $\M$ is $\kappa$-saturated if and only if for any regular epimorphism in $\Pro(\defT(T))$, whose codomain has variable length strictly less than $\kappa$, the interpretation in $\M$ is a surjection. Now, hyperimaginaries are roughly speaking quotients of types. If we wish to study them categorically, we would need a category of types. Hence, $\Pro(\defT(T))$ is exactly the right setting to study hyperimaginaries.

\smallskip

Although stability is arguably the poster child of classification theory, a slightly weaker property---simplicity---introduced by Shelah \cite{SHELAH1980177}, also received attention for its nice behaviours, culminating in the celebrated Kim-Pillay Theorem \cite{KIM1997149}. One thing that model-theorists noticed was that, while working in stable theories usually requires one to deal with imaginaries, working in simple theories requires one to deal with hyperimaginaries, which are heuristically `type-definable imaginaries'. More precisely, a hyperimaginary is an equivalence class of a type-definable equivalence relation. We say a theory eliminates hyperimaginaries if every hyperimaginary is interdefinable with a sequence of imaginaries; this is equivalent to saying that every type-definable equivalence relation is the intersection of some definable equivalence relations. The first mention of this is in \cite{Pillay_Poizat_1987}, where Pillay and Poizat proved that stable theories eliminate hyperimaginaries. In \cite{BuechlerSteven2001St}, this was generalised to supersimple theories, yet whether simple theories eliminate hyperimaginaries remains an open question. 

\smallskip

So far, the model theory and the category theory are well-known, though perhaps only to model theorists and categorical theorists respectively. In order to make this paper accessible to both communities, we include basic definitions, as well as proofs for elementary facts. We hope this will be helpful in further bridging the two areas.

\smallskip

We come to the main original contributions of the paper in \S\S\ref{hyperimaginaries}, \ref{heq}. We are interested in whether $\Pro(\defT(T))$ is exact. Now, we shall see in Proposition~\ref{if prodef exact then def exact} that, if the quotient of a definable equivalence relation is type-definable, then the quotient itself is also definable. In other words, if we want $\Pro(\defT(T))$ to be exact, then we should require $\defT(T)$ to also be exact. Hence, we might as well assume $T$ eliminates imaginaries, or work over $\defT(T^{eq})$, equivalently $\defT(T)_{ex/reg}$. Assuming this, elimination of hyperimaginaries allows us to decompose a type-definable equivalence relation as a cofiltered limit of definable equivalence relations, each of which has a definable quotient. This allows us to obtain in Theorem~\ref{eliminate hyper iff pro exact} the following characterisation:
\[T \textup{ eliminates hyperimaginaries } \Leftrightarrow \ \Pro(\defT(T)_{ex/reg}) \textup{ is exact.}\]

Similar to the eq construction for imaginaries, there is an heq construction for hyperimaginaries. We study this in \S\ref{heq}. The construction comes from \cite{Hart2000-HARCAC-24}, but unlike the eq construction, hyperimaginaries cannot be na\"{i}vely added to a structure as real elements. Indeed, \cite[Example 3.1.6]{alma992983361725701631} shows that the quotient of a type-definable equivalence relation can have bounded infinite cardinality. If we simply add the quotient as a new sort, by compactness there are elementarily equivalent structures, whose interpretations of that sort have unbounded cardinalities. Moreover, one cannot add the quotient map as a new function symbol, since the sort of the domain might be of infinite length. Nevertheless, \cite{Hart2000-HARCAC-24} constructs $\U^{heq}$ from a monster model $\U$ as an expansion with quotients of type-definable equivalence relations as new sorts. Instead of considering $\U^{heq}$ as a first-order structure in a certain language, one considers how automorphisms of $\U$ extends to automorphisms of $\mathcal{U}^{heq}$ and, using this, defines the various closure operators like dcl, acl and bbd.

\smallskip

Shortly afterwards, Ben-Yaacov started developing positive model theory and noted in \cite[Example 2.16]{Ben-Yaacov2003-BENPMT} that $T^{heq}$ can be given a syntax. This was further developed by Dobrowolski and Kamsma \cite{DobrowolskiJan2022Kipl} to study NSOP$_1$ in positive logic. Omitting negation, as in positivie logic, turns out to be the correct choice for studying hyperimaginaries, since type-definable sets are not closed under complement. In \S\ref{heq}, we will argue that the correct logic for studying hyperimaginaries is a logic with infinite conjunctions, infinite existential quantification, finite disjunction, and no negation. This is the internal logic of $\infty$-coherent categories, such as $\Pro(\defT(T))$. Using this, we will define an infinitary syntactic category associated to $\U^{heq}$, denoted by $\tpdefT(\U^{heq})$, and prove that 
\[\tpdefT(\U^{heq}) \simeq (\Pro(\defT(T)))_{ex/reg}. \] 

We summarise our results.

\begin{mainthm}
    \input{6main_theorem}

\end{mainthm}

\smallskip

\noindent\textbf{Acknowledgements.} This work is part of my PhD research, done under the supervision of Omar Le\'on S\'anchez and the co-supervision of Nicola Gambino. I would like to thank both for their guidance. I would also like to thank Mark Kamsma, who posed the question that this paper is answering, as well as Giacomo Tendas for helpful discussions.

%% file: 6main_theorem.tex
    Let $T$ be a consistent complete first-order theory and let $\U$ be a monster model.
    \begin{enumerate}
    \item $T$ eliminates imaginaries if and only if $\defT(T)$ is exact.

    \item $T$ eliminates hyperimaginaries if and only if $\Pro(\defT(T)_{ex/reg})$ is exact.
        
    \item $T$ eliminates imaginaries and hyperimaginaries if and only if $\Pro(\defT(T))$ is exact.

    \item We have the following equivalences of categories: 
        \begin{align*}
            \defT(T^{eq}) &\simeq \defT(T)_{ex/reg}, \\
            \tpdefT(\U^{heq}) &\simeq (\Pro(\defT(T)))_{ex/reg}.
        \end{align*}
    \end{enumerate}

%% file: 2eq_and_exact.tex
\bigskip

\section{Imaginaries and exactness}\label{preliminaries}
In this section, we establish some preliminaries and review the connection between exactness and imaginaries.

\medskip

\subsection{Regularity}
Let $T$ be a consistent complete first-order theory. (Much of this can be done in the more general setting of positive or coherent logic, which is the fragment of first-order logic consisting of truth $\top$, falsity $\bot$, finite disjunction $\vee$, finite conjunction $\wedge$, equality $=$, and existential quantification $\exists$, or regular logic, which consists of $\top,\wedge,=,$ and $\exists$.) Let $\mathcal{U}$ be a monster model of $T$. We define the category associated with $T$, known to category theorists as the syntactic category and to model theorists as the category of definable sets and functions. By definable, we mean 0-definable.

\begin{defin}
    We define the category $\defT(T)$ as follows:
    \begin{itemize}
        \item Objects are definable subsets of $\mathcal{U}^n$ for any $n\in\omega$;
        \item Morphisms are definable functions.
    \end{itemize} 
    
    Given a formula $\phi(x)$, we write $[\phi(x)]$ for the object associated with it. 

    Equivalently, the objects are formulas modulo $T$, and the morphisms are formulas, which $T$ proves to be a function, modulo $T$.
\end{defin}

We note here that $\defT(T)$ contains exactly two objects in the empty sort, namely $[\top]$ and $[\bot]$ considered as the two subsets of $\mathcal{U}^0 = \{ \ast \}$. 

The category $\defT(T)$ inherits much of the logical structure of $T$, in the sense that $T$ is the interal logic of $\defT(T)$. For an exposition on this topic, see \cite{Butz1998}.

\begin{defin}
Let $\C$ and $\mathcal{D}$ be categories. Let $f:X\rightarrow Y$ be a morphism in $\C$.
\begin{enumerate} 
    \item The \textit{kernel pair} of $f$ is the pair $(p,q)$ of morphisms obtained by pulling back $f$ with itself.
    \[\begin{tikzcd}
    	P & X \\
    	X & Y
    	\arrow["p", from=1-1, to=1-2]
    	\arrow["q"', from=1-1, to=2-1]
    	\arrow["\lrcorner"{anchor=center, pos=0.125}, draw=none, from=1-1, to=2-2]
    	\arrow["f", from=1-2, to=2-2]
    	\arrow["f"', from=2-1, to=2-2]
    \end{tikzcd}\]
    \item The morphism $f$ is a \emph{regular epimorphism} if there is a parallel pair of morphisms $g,h:W\rightrightarrows X$ such that $f$ is the coequaliser of $(g,h)$. This implies that $f$ is an epimorphism.
    \item The category $\C$ is \textit{regular} if it has all finite limits, coequalisers of kernel pairs exist, and regular epimorphisms are stable under pullback.
    \item Suppose $\C$ and $\mathcal{D}$ are regular. A functor $F:\C\rightarrow\mathcal{D}$ is \emph{regular} if $F$ preserves finite limits and regular epimorphisms. 
    \item We denote by $\textup{Reg}(\C,\mathcal{D})$ the full subcategory of the functor category $[\C,\mathcal{D}]$ spanned by regular functors.
\end{enumerate}
\end{defin}

The definition of regular epimorphisms might seem a bit abstract, but in a regular category, they are intuitively surjections and quotients.

\begin{lem}[{\cite[Lemma 2.3, Proposition 2.4]{Butz1998}}]\label{basics of regular}
Let $\C$ be a regular category.
    \begin{enumerate}
        \item A regular epimorphism is the coequaliser of its kernel pair.
        \item A morphism is a regular epimorphism and a monomorphism if and only if it is an isomorphism.
        \item In a regular category, each morphism $f$ can be factored as a regular epimorphism $e$ followed by a monomorphism $m$. This factorisation is unique up to isomorphism.
        \[\begin{tikzcd}
        	& W & \\
        	X && Y
        	\arrow["m", hook, from=1-2, to=2-3]
        	\arrow["e", two heads, from=2-1, to=1-2]
        	\arrow["f"', from=2-1, to=2-3]
        \end{tikzcd}\]
    \end{enumerate}
\end{lem}

In the diagram above, the monomorphism $m:W\rightarrow Y$ is called the \emph{image} of $f$. For example, in the regular category of \textbf{Set}, it is precisely the set-theoretic image.

The internal logic of regular categories is known as regular logic, which is the fragment of first-order logic consisting of $\top,\wedge,=,$ and $\exists$; see \cite[\S 3]{Butz1998}. For example, the terminal object represents $\top$, binary products represent $\wedge$ between two formulas with disjoint free variables, equalisers represent $=$, and images represent $\exists$.

We prove a series of lemmas about $\defT(T)$, which will culminate in a characterisation of its regular epimorphisms and a proof that it is a regular category. 

\begin{prop}\label{regular defT}
In the category $\defT(T)$, the following hold:
    \begin{enumerate}
        \item Finite limits exist.
        \item If a morphism is provably surjective, then it is a regular epimorphism.
        \item Kernel pairs have coequalisers.
        \item A morphism is provably injective if and only if it is a monomorphism.
        \item If a morphism is a regular epimorphism, then it is provably surjective.
        \item Regular epimorphisms are stable under pullbacks.
    \end{enumerate}
\end{prop}

\begin{proof}
    \begin{enumerate}
        \item It is routine to check that $\top$ gives the terminal object, and $\wedge$ between two formulas with disjoint variables gives the product. Given a parallel pair of morphisms $f,g:[\phi(x)]\rightarrow[\psi(y)]$, the equaliser is given by the inclusion $[\phi(x)\wedge f(x)=g(x)]$.
        
        \item Let $f:[\phi(x)]\rightarrow[\psi(y)]$ be a definable surjection. If $f$ is to be a regular epimorphism, it must be the coequaliser of its kernel pair $(\pi_1,\pi_2)$. Suppose $g:[\phi(x)]\rightarrow [\chi(z)]$ is a cocone for $(\pi_1,\pi_2)$.
        \[\begin{tikzcd}
        	{[\phi(x_1)\wedge\phi(x_2)\wedge f(x_1)=f(x_2)]} & {[\phi(x)]} & {[\psi(y)]} \\
        	&& {[\chi(z)]} \\
        	&& {}
        	\arrow["{\pi_2}"', shift right, from=1-1, to=1-2]
        	\arrow["{\pi_1}", shift left, from=1-1, to=1-2]
        	\arrow["f", from=1-2, to=1-3]
        	\arrow["g"', from=1-2, to=2-3]
        	\arrow["h", dashed, from=1-3, to=2-3]
        \end{tikzcd}\]
        We define the map $h$ as 
        \[ \psi(y)\wedge\chi(z)\wedge\exists x \ (\phi(x)\wedge f(x)=y\wedge g(x)=z).\]
        This is total, because $f$ is surjective. This is single-valued, because $g$ is a cocone for $(\pi_1,\pi_2)$. It is unique, because every $y$ has a preimage $x$ under $f$, and the commuting condition enforces that $h(y)=g(x)$.

        \item Given a morphism $f:[\phi(x)]\rightarrow[\psi(y)]$, the coequaliser of its kernel pair is given by $f: [\phi(x)]\rightarrow[\psi(y)\wedge\exists x \ (\phi(x)\wedge f(x)=y)]$. The proof of this is similar to above.

        \item Suppose $f:[\phi(x)]\rightarrow[\psi(y)]$ is not a monomorphism, so there exist $g,h:[\chi(z)]\rightrightarrows[\phi(x)]$ such that $g\neq h$ and $fg=fh$. Then, $T$ proves that 
        \[ \exists z \ (\chi(z)\wedge g(z)\neq h(z) \wedge fg(z)=fh(z)).\]
        In particular, $T$ proves that $f$ is not injective.

        Conversely, suppose $f$ is not injective. Then, the following diagram observes that $f$ is not a monomorphism.
        \[\begin{tikzcd}
        	{[\phi(x_1)\wedge\phi(x_2)\wedge f(x_1)=f(x_2)]} & {[\phi(x)]} & {[\psi(y)]}
        	\arrow["{\pi_2}"', shift right, from=1-1, to=1-2]
        	\arrow["{\pi_1}", shift left, from=1-1, to=1-2]
        	\arrow["f", from=1-2, to=1-3]
        \end{tikzcd}\]
                
        \item Let $f:[\phi(x)]\rightarrow[\psi(y)]$ be a regular epimorphism. Examining the proof of Lemma \ref{basics of regular} in \cite{Butz1998}, we see that part (1) only requires $\C$ to have finite limits and coequalisers of kernel pairs. We have already proved this for $\defT(T)$. Hence, $f$ is the coequaliser of its kernel pair $(\pi_1,\pi_2)$. Factorise $f$ as $me$ as below, where $e$ is $f$ with  codomain restricted to its image and $m$ is inclusion. 
        \[\begin{tikzcd}
        	{[\phi(x_1)\wedge\phi(x_2)\wedge f(x_1)=f(x_2)]} & {[\phi(x)]} & {[\psi(y)]} \\
        	&& {[\psi(y)\wedge\exists x \ (\phi(x)\wedge f(x)=y)]} \\
        	&& {}
        	\arrow["{\pi_2}"', shift right, from=1-1, to=1-2]
        	\arrow["{\pi_1}", shift left, from=1-1, to=1-2]
        	\arrow["f", two heads, from=1-2, to=1-3]
        	\arrow["e"{description}, two heads, from=1-2, to=2-3]
        	\arrow["g", shift left, curve={height=-6pt}, dashed, from=1-3, to=2-3]
        	\arrow["m", hook, from=2-3, to=1-3]
        \end{tikzcd}\]
        Note that $e$ is surjective, so $e$ is a regular epimorphism. Also, $m$ is provably injective, so $m$ is a monomorphism. We calculate that $me\pi_1=f\pi_1=f\pi_2=me\pi_2$, so canceling $m$ on the left shows that $e$ is a cocone for $(\pi_1,\pi_2)$. The universality of the coequaliser induces the map $g$ such that $gf=e$. Then, we cancel $f$ on the right from $mgf=f$ to obtain $mg=\textup{id}$. Similarly, cancel $e$ on the right from $gme=e$ to obtain $gm=\textup{id}$.
        
        Since $m$ is an isomorphism and is the inclusion map, we conclude that $m$ is the identity, and $f=e$. Now, $e$ is certainly surjective. 

        \item That a regular epimorphism is stable under pullback is known as substitution. In other words, if $f:[\phi(x)]\rightarrow[\psi(y)]$ is a regular epimorphism and $g:[\chi(z)]\rightarrow[\psi(y)]$ is any morphism with the same codomain, then we require that the projection $[\phi(x)\wedge\chi(z)\wedge f(x)=g(z)]\rightarrow [\chi(z)]$ is surjective. This is certainly true given that $f$ is surjective.
    \end{enumerate}
\end{proof}

\begin{cor}
    The category $\defT(T)$ is regular. The regular epimorphism-monomorphism factorisation of a morphism $f:[\phi(x)]\rightarrow[\psi(y)]$ is given by the diagram below, where $e$ is $f$ with restricted domain and $m$ is the inclusion map.
\[\begin{tikzcd}
	& {[\psi(y)\wedge\exists x \ (\phi(x)\wedge f(x)=y)]} & \\
	{[\phi(x)]} && {[\psi(y)]}
	\arrow["m", hook, from=1-2, to=2-3]
	\arrow["e", two heads, from=2-1, to=1-2]
	\arrow["f"', from=2-1, to=2-3]
\end{tikzcd}\]
\end{cor}

\begin{example}\label{degenerate example}
    There are examples of $T$ with non-regular epimorphisms. By part (2) of Proposition~\ref{regular defT}, they are in particular not surjective. Let $T$ be the theory of the singleton set in the empty language. Then, every definable set is either the empty set or the singleton set, so $\defT(T)$ is equivalent to the category $\mathbbm{2} = \{ \bot\rightarrow \top\}$. The unique arrow from the initial object to the terminal is an epimorphism, but not a regular epimorphism. 
\end{example}

\begin{prop}
    All epimorphisms in $\defT(T)$ are regular if and only if $T$ has at least two distinct definable points in the same sort.
\end{prop}

\begin{proof}
    $(\Rightarrow)$ \quad Suppose all epimorphisms are regular. In particular, since the unique morphism $!:\bot \rightarrow \top$ is not a regular epimorphism, it is not an epimorphism. Hence, there is a parallel pair $a,b:\top\rightrightarrows X$ such that $a\neq b$. These are precisely definable points of sort $X$.

\smallskip

    $(\Leftarrow)$ \quad Suppose $T$ has two distinct definable points $a$ and $b$ of the sort $S$. Suppose a definable function $f:[\phi(x)]\rightarrow[\psi(y)]$ is not a regular epimorphism, so in particular not surjective. Then, we can define two parallel maps $g,h:[\psi(y)]\rightrightarrows S$. The map $g$ sends everything to $a$. The map $h$ sends elements in the image of $f$ to $a$ and the rest to $b$. These maps show that $f$ is not an epimorphism. 
\end{proof}

\medskip

\subsection{The eq construction}
We recall a few basic facts about elimination of imaginaries and the eq construction. See \cite[\S8.4]{alma9932857994401631} for reference.

\begin{defin}
\begin{enumerate}
    \item     We say a theory $T$ has \textit{(uniform) elimination of imaginaries} if definable equivalence relations have definable quotients: given a definable equivalence relation $E(x,y)$ on $\U^n$ for any $n\in\omega$, there is a definable set $\U^n/E$ with a definable function $f:\U^n\rightarrow \U^n/E$, such that $f(x)=f(y) \Leftrightarrow E(x,y)$ for all $x,y\in \U^n$.
    \item     Let $\mathcal{U}^{eq}$ be the following expansion of $\mathcal{U}$. For every definable equivalence relation $E$, there is a new sort $S_E$, whose interpretation is the equivalence classes of $E$. There is a new function symbol $\pi_E$ from the home sort $\mathcal{U}$ to $S_E$, which is interpreted by the quotient map of $E$. Let $T^{eq}$ be the the theory of $\mathcal{U}^{eq}$.
\end{enumerate}
\end{defin}

The definition of elimination of imaginaries can be equivalently stated with definable equivalence relations $E$ over any definable set $X$. Indeed, given a definable equivalence relation $E$ over $X$, one can extend $E$ to an equivalence relation over $\mathcal{U}^n$, given by $E'(x,y) := E(x,y) \vee x=y$. Then, the quotient $X/E$ is recovered as the image of the map $X\hookrightarrow\mathcal{U}\rightarrow\mathcal{U}/E'$.

\begin{prop}[{\cite[Proposition 8.4.5]{alma9932857994401631}}]
\begin{enumerate}
    \item A subset of the home sort is definable in $T^{eq}$ if and only if it is definable in $T$. 
    \item The theory $T^{eq}$ eliminates imaginaries.
\end{enumerate}
\end{prop}

\medskip

\subsection{The exact completion}
We recall a few facts about exact categories and how a regular category can be completed to an exact category freely. See \cite{CarboniA.1998Raec, LackStephen1999Anot} for details.

\begin{defin}
    A category is \textit{(Barr-)exact} if it is regular and every internal equivalence relation $E\rightrightarrows X$ has a coequaliser $X\rightarrow X/E$, whose kernel pair is the equivalence relation $E\cong X\times_{X/E}X$.
\end{defin}

\begin{prop}\label{EI iff exact}
    The theory $T$ eliminates imaginaries if and only if $\defT(T)$ is exact.
\end{prop}

\begin{proof}
    $(\Rightarrow)$  \quad Suppose $T$ eliminates imaginaries. Let $\pi_1,\pi_2:E\rightrightarrows X$ be an internal equivalence relation in $\defT(T)$. Then, the relation $E'$ defined by
    \[x_1,x_2\in X \wedge (\exists x'_1,x'_2 \ E(x'_1,x'_2)\wedge \pi_1(x'_1)=x_1\wedge \pi_2(x'_2)=x_2)\]
    is an equivalence relation on $X$. Let $f:X\rightarrow Y$ be its definable quotient. We prove that this is the coequaliser of $E\rightrightarrows X$ and that $E$ is the kernel pair.

    Given a cocone $g:X\rightarrow W$ for the parallel pair $(\pi_1,\pi_2)$, the unique factorising map $Y\rightarrow W$ is given by 
    \[ \exists x\in X \ (f(x)=y\wedge g(x)=w).\]
    Given a cone $p_1,p_2:V\rightrightarrows X$ such that $fp_1=fp_2$, the unique factorising map $V\rightarrow E$ is given by
    \[ p_1(v)=\pi_1(e)\wedge p_2(v)=\pi_2(e). \]

\smallskip

    $(\Leftarrow)$  \quad Suppose $\defT(T)$ is exact. Let $E$ be a definable equivalence relation over $X$. Then the projections $\pi_1,\pi_2: E\rightrightarrows X$ form an internal equivalence relation in $\defT(T)$. Let $f:X\rightarrow X/E$ be its coequaliser. This gives the definable quotient, since $E$ is isomorphic to the kernel pair of $f$, which is given by the projections $[x_1,x_2\in X \wedge f(x_1)=f(x_2)]\rightrightarrows X$. 
\end{proof}

\begin{defin}
The \textit{exact completion} of a regular category $\mathcal{C}$, denoted as $\mathcal{C}_{ex/reg}$, is the exact category equipped with a regular embedding $\mathcal{C}\rightarrow\mathcal{C}_{ex/reg}$, which induces for each exact category $\mathcal{D}$ an equivalence 
\[\textup{Reg}(\mathcal{C},\mathcal{D})\simeq\textup{Reg}(\mathcal{C}_{ex/reg},\mathcal{D}).\]
\end{defin}

That the exact completion exists and its explicit constructions can be found in \cite{CarboniA.1998Raec} and \cite{LackStephen1999Anot}. We recall the description in \cite{CarboniA.1998Raec}. Within a regular category, one can define a calculus of relations. Relations from $X$ to $Y$ are subobjects of $X\times Y$, so are ordered by the usual subobject order. The composition of relations can be described as a combination of pullbacks and projections, but it is perhaps more useful to describe it as follows. If $R(x,y)$ and $S(y,z)$ are relations, their composition $SR(x,z)$ is defined as $\exists y \ R(x,y) \wedge S(y,z)$. Moreover, any relation $R(x,y)$ has an involution $R^\circ(y,x)$. This is defined by $R^\circ(y,x)\leftrightarrow R(x,y)$. 

\begin{defin}[{\cite[Definition 11, Proposition 12]{CarboniA.1998Raec}}]\label{ex/reg carboni}
    Let $\C$ be a regular category. The exact completion of $\C$ can be described as follows:
    \begin{itemize}
        \item An object of $\C_{ex/reg}$ is an equivalence relation $E\hookrightarrow X\times X$ in $\C$;
        \item A morphism $R: (X,D) \rightarrow (Y,E)$ is a relation $R\hookrightarrow X\times Y$ satisfying
            \[ RD = R = ER \]
        and 
            \[ D \leq R^\circ R \textup{ and } RR^\circ \leq E. \]
    \end{itemize}
\end{defin}

\begin{thm}\label{Teq is exact}
    There is an equivalence of categories $\defT(T^{eq}) \simeq \defT(T)_{ex/reg}$.
\end{thm}

By Proposition \ref{EI iff exact}, $\defT(T^{eq})$ is exact. We use the following lemma.

\begin{lem}\label{F defT defTeq is regular}
        The inclusion functor $F: \defT(T) \hookrightarrow \defT(T^{eq})$ is a regular functor.
\end{lem}
    \begin{proof}
        It is clear that $F$ preserves finite limits, since it preserves $\top$ and $\wedge$. 

        Now, if $f$ is a regular epimorphism in $\defT(T)$, then $T$ proves that $f$ is a surjection. $T^{eq}$ extends $T$, so $T^{eq}$ also proves that $f$ is a surjection. Hence, $f$, or more precisely, $F(f)$, is a regular epimorphism in $\defT(T^{eq})$.
    \end{proof}

\begin{proof}[Proof of Theorem \ref{Teq is exact}]
By the universal property of $\defT(T)_{ex/reg}$, the inclusion functor $F$ from Lemma \ref{F defT defTeq is regular} induces a regular functor $\overline{F}: \defT(T)_{ex/reg} \rightarrow \defT(T^{eq})$. We prove that $\overline{F}$ is fully faithful and essentially surjective.

\smallskip

Full: \quad In $T$, let $R,S$ be equivalence relation over $X,Y$ respectively, so that $\langle X,R \rangle $ and $\langle Y,S \rangle$ are objects of $\defT(T)_{ex/reg}$. Let $R'$ be any definable equivalence relation over $\U^n$ extending $R$, and likewise for $S'$. Recall that $\pi_{R'}: \U^n\rightarrow \U^n/R'$ is a function symbol in $T^{eq}$, and that the image $\pi_{R'}(X)$ of $X$ under this function is the definable quotient of $R$. As $\overline{F}$ is regular, $F(\langle X,R \rangle) = \pi_{R'}(X)$ and $F(\langle Y,S \rangle) = \pi_{S'}(Y)$.

Let $\tilde{\theta}(\tilde{x},\tilde{y})$ be a morphism from $\pi_{R'}(X)$ to $\pi_{S'}(Y)$. Define $\theta(x,y)$ on the home sorts by $\tilde{\theta}(\pi_{R'}(x),\pi_{S'}(y))$. Note as $\theta$ is defined on the home sorts, it is a relation on $X$ and $Y$ in $\defT(T)$. We check that $\theta$ is in fact a morphism in $\defT(T)_{ex/reg}$ from $\langle X,R\rangle$ to $\langle Y,S\rangle$.

Suppose $a,b\in X$ and $c\in Y$ satisfy $R(a,b)$ and $\theta(b,c)$. Then, $\pi_{R'}(a)=\pi_{R'}(b)$ and $\tilde{\theta}(\pi_{R'}(b),\pi_{S'}(c))$ hold, so that $\theta(a,c)$. Hence, we have $\theta R \subseteq \theta$. Conversely, the reflexivity of $R$ implies $\theta\subseteq\theta R$. Similarly, one can prove $\theta = S\theta$.

Suppose $R(a,b)$ holds. Let $c\in Y$ be such that $\tilde\theta (\pi_{R'}(a))=\pi_{S'}(c)$. Then, we have $\theta(a,c)$. Since $\pi_{R'}(a)=\pi_{R'}(b)$, we also have $\theta(b,c)$. Hence, $R\leq\theta^\circ\theta$. Similarly, we have $\theta\theta^\circ\leq S$.

Moreover, $\theta$ considered as an object is the pullback of $\theta R: \langle X,1_X\rangle \rightarrow \langle Y,S\rangle$ and $S:\langle Y,1_Y\rangle \rightarrow \langle Y,S\rangle$. This pullback is preserved by $\overline{F}$, so $\theta(x,y)$ holds if and only if $\overline{F}(\theta)(\pi_R(x)) = \pi_S(y)$. Hence, $\overline{F}(\theta) = \tilde{\theta}$.

\smallskip

Faithful: \quad Let $\theta,\mu: \langle X,R\rangle \rightrightarrows \langle Y,S\rangle$ be two morphisms in $\defT(T)_{ex/reg}$. Then, $\theta,\mu$ are relations on definable subsets of the home sort, so are objects in $\defT(T)$. Since $\overline{F}$ is regular, $\overline{F}(\theta)=\overline{F}(\mu)$ as morphisms if and only if $\overline{F}(\theta)=F(\theta)=F(\mu)=\overline{F}(\mu)$ as objects. Since $T$ and $T^{eq}$ prove the same sentences on the home sort, this occurs if and only if $\theta=\mu$.

\smallskip

Essentially surjective: \quad Let $[\phi(\tilde{x})] \in \defT(T^{eq})$, with the tuple $\tilde{x}$ having sort $\U^n/E$. Define $\psi(x)$ as $ \phi(\pi_{E}(x))$. This is a formula in the home sort, so $[\psi(x)]\in\defT(T)$. Let $E'$ be the equivalence relation of $E$ restricted to $\psi(x)$, so that $\langle\psi(x),E'\rangle\in\defT(T)_{ex/reg}$. Since $\pi_E\upharpoonright_{\psi(x)}:[\psi(x)]\rightarrow[\phi(\tilde x)]$ is the coequaliser of its kernel pair $E'\rightrightarrows[\psi(x)]$, the regularity of $\overline{F}$ implies $\overline{F}(\langle \psi(x), E'\rangle) \cong [\phi(\tilde x)]$.

\end{proof}

The exact completion is intimately related to the pretopos completion. Indeed, Makkai and Reyes \cite{alma992977584040101631} mainly worked on pretoposes instead of exact categories. For the rest of this section, we will study when the pretopos completion and the exact completion of $\defT(T)$ coincide. For a more general approach, see \cite{ShulmanMichael2012Ecas}.

\begin{defin}
\begin{enumerate}

    \item A \textit{coherent} category is a regular category, such that its subojects have finite joins $\vee$ (and thus $\textup{Sub} X$ is an upper semilattice for any $X\in\C$), and for any morphism $f:X\rightarrow Y$, the map $f^{-1}:\textup{Sub} Y \rightarrow \textup{Sub} X$ is a homomorphism of upper semilattices.
    \item Let $\C$ and $\mathcal{D}$ be coherent categories. A functor $F:\C\rightarrow\mathcal{D}$ is \emph{coherent} if it is regular and preserves finite joins. We write $\textup{Coh}(\C,\mathcal{D})$ for the full subcategory of $[\C,\mathcal{D}]$ spanned by the coherent functors.
    \item A coherent category $\C$ is a \textit{pretopos} if it is exact, and has finite disjoint coproducts stable under pullback.
    \item The \textit{pretopos completion} of a coherent category $\mathcal{C}$, denoted as $\overline{\mathcal{C}}$, is the pretopos with a coherent embedding $\mathcal{C}\rightarrow\overline{\mathcal{C}}$, inducing for each pretopos $\mathcal{D}$ an equivalence 
    \[\textup{Coh}(\mathcal{C},\mathcal{D}) \simeq \textup{Coh}(\overline{\mathcal{C}},\mathcal{D}).\]
\end{enumerate}
\end{defin}

Note that $\defT(T)$ is coherent, where joins are given by disjunctions. Since $\overline{\defT(T)}$ is exact, there is a natural map $\defT(T)_{ex/reg}\rightarrow\overline{\defT(T)}$. We investigate when this functor is an equivalence of categories.

\begin{example}\label{degenerate example2}
    Recall Example \ref{degenerate example}, where $T$ is the theory of a singleton set in the empty language. All definable equivalence relations are trivial, so in fact $T$ eliminates imaginaries and $\defT(T)\simeq\defT(T)_{ex/reg}$ is exact. However, every definable set is either empty or a singleton set, so $T$ does not define the disjoint sum of $\{\ast\}$ with itself, which has two elements. Therefore, $\defT(T)$ is not even a pretopos, let alone the pretopos completion.
\end{example}

It turns out this is the only counterexample. We use the following theorem. An object is \textit{well-supported} if the unique arrow into the terminal is a regular epimorphism. An object $D$ is \textit{decidable} if the diagonal $\Delta_D:D\rightarrow D\times D$ has a complement.

\begin{thm}[{\cite[Proposition 5.1]{EmmeneggerJacopo2020Edac}}]\label{EPR}
    Let $\mathcal{C}$ be coherent. Suppose also that 
    \begin{enumerate}
        \item For all $A\in\mathcal{C}$, there is a well-supported object $B$ and a monomorphism $m:A\hookrightarrow B$.
        \item There exists a decidable object $D$ such that the complement of the diagonal $\Delta_D:D\rightarrow D\times D$ is well-supported.
    \end{enumerate}
    Then, the functor $\mathcal{C}_{ex/reg}\rightarrow\overline{\mathcal{C}}$ is an equivalence of categories.
\end{thm}

\begin{thm}\label{degenerate theorem}
    The following are equivalent:
    \begin{enumerate}
        \item There is a sort $S$, such that $T\vdash \exists x,y:S \  x\neq y$.
        \item The functor $\defT(T)_{ex/reg}\rightarrow\overline{\defT(T)}$ is an equivalence of categories.
        \item The category $\defT(T^{eq})$ is a pretopos.
    \end{enumerate}
\end{thm}

\begin{proof}
    ($1\Rightarrow2$) In $\defT(T)$, a regular epimorphism is a definable surjection, an object is well-supported if it is non-empty, and every object is decidable. Therefore, condition (1) in Theorem~\ref{EPR} is always satisfied. An object satisfies condition (2) if it has more than one element, which is the assumption.

    ($2\Rightarrow3$) This follows from Theorem \ref{Teq is exact}.

    ($3\Rightarrow1$) The contraposition is precisely the multi-sorted version of Example \ref{degenerate example2}.
\end{proof}

%% file: 3Basic_properties_of_pro.tex
\bigskip

\section{Types and the pro-completion}\label{pro}
In this section, we study the pro-completion of $\defT(T)$. In the category theory literature, Barr \cite{BARR1986113} studied $\Pro (\C)$ in the case when $\C$ is regular. Textbooks on shape theory also contain references for the pro-completion, such as \cites[Ch.~I \S1\addsemicolon]{alma992985533896801631}[Appendix]{alma992977584035101631}. 

\medskip

\subsection{Types and morphisms of types}
We first give the model theoretic description of $\Pro(\defT(T))$ as the category of types, which was first studied in \cite{KAMENSKY2007180}. Continue to fix $T$ to be a consistent complete first-order theory with monster model $\mathcal{U}$. Fix a suitably big cardinal $\kappa$ and say a set is small if its cardinality is strictly less than $\kappa$. When we say `sufficiently large', we mean larger than $\kappa$.

\begin{defin}
\begin{enumerate}
    \item A \emph{type} is a $T$-deductively-closed set of formulas in a small set of variables (not necessarily complete). To emphasise the categorical approach, we write $X, Y$ for a type (rather than the model-theoretic convention of $p,q$). We write $X_i$ for the formulas in the type $X$, and identify $(X_i)_{i\in I}$ with $X$ itself. The notation is justified by the fact that $X=\lim X_i$ in $\Pro(\defT(T))$.
    \item We define the \emph{category of types} as follows:
    \begin{itemize}
        \item Objects are types and falsity $\bot$;
        \item A morphism $f:(X_i)_{i\in I}\rightarrow (Y_j)_{j\in J}$ is induced by a collection of definable functions $f_j: X_{i_j} \rightarrow Y_j$ for each $j\in J$, which is compatible in the following sense: for any $j,j'\in J$,
    
        \[ T \cup X(x) \vdash f_j(x)=f_{j'}(x) \]

        We then take a quotient on the set of compatible collections of definable functions by the following equivalence relation: for any two compatible collections $(f_j)_{j\in J}, (g_j)_{j\in J}$ from $X$ to $Y$,

        \[(f_j)\sim (g_j) \ \Leftrightarrow \ T \cup X(x) \vdash f_j(x)=g_j(x) \textup{ for all }j. \]
        Hence, more precisely a morphism is an equivalence class of compatible collections of definable functions.
    \end{itemize}
\end{enumerate}
\end{defin}

\begin{remark}
    We require a type to be deductively closed, so that if two morphisms in the category of types are $T$-equivalent, then they are the same morphism. In particular, this ensures that there is only one identity morphism for each object, and composition is well-defined. Having said that, we might define a specific type by giving a non-deductively-closed set of formulas. Formally, we mean the deductive closure of the given set.
    
    Unpacking the definition of a morphism in the special case when the domain or the codomain is $\bot$, one sees that $\bot$ is the initial object.

    By compactness, a collection of definable functions is compatible if and only if for any $j,j'\in J$, there is some $k\in J$ such that $T\cup X_k(x)\vdash f_j(x) = f_{j'}(x)$. Likewise, two compatible collections $(f_j),(g_j)$ are equivalent if and only if for all $j$, there is some $k$ such that $T\cup X_k(x)\vdash f_j(x)=g_j(x)$. Compare this to \cite[Ch.~I, pp. 4--6]{alma992985533896801631}.

    A compactness argument also shows that a morphism is precisely a pro-definable set, whose interpretation in a suitably saturated $\U$ is the graph of a function. This is a result of \cite{KAMENSKY2007180}, which is recorded here as a lemma.
\end{remark}

\begin{lem}[{\cite[Corollary 9]{KAMENSKY2007180}}]\label{functions of Pro}
    Let $\U\models T$ be sufficiently saturated. Let $X,Y\in\Pro(\defT(T))$. There is a natural bijection between $\Hom(X,Y)$ and the subobjects of $X\times Y$, whose interpretation in $\mathcal{U}$ is a function from $X(\mathcal{U})$ to $Y(\mathcal{U})$
\end{lem}

\smallskip

We recall the categorical construction of $\Pro (\C)$.
\begin{defin}
\begin{enumerate}
    \item A small category $I$ is \emph{cofiltered} if it has finite cones.
    \item A \emph{cofiltered limit} in $\C$ is a limit whose diagram is given by a cofiltered category.
    \item Given a cofiltered diagram $F:I\rightarrow\C$, the \emph{transition maps} of $F$ are the morphisms $F\alpha:Fi\rightarrow Fj$ for morphisms $\alpha:i\rightarrow j$ in $I$.
    \item Given a category $\C$, the \emph{pro-completion} $\Pro(\C)$ is a category with cofiltered limits, equipped with a functor $\C\rightarrow\Pro(\C)$ inducing, for any $\mathcal{D}$ with cofiltered limits, an equivalence of categories
        \[ \Fun(\C,\mathcal{D}) \simeq \textup{Cofilt}(\Pro(\C),\mathcal{D}),\]
        where $\textup{Cofilt}$ denotes the full subcategory of the functor category spanned by cofiltered-limit-preserving functors.
\end{enumerate}
\end{defin}

The universal property defines $\Pro(\C)$ up to equivalence. When $\C$ has all finite limits, $\Pro(\mathcal{C})$ is equivalent to $\textup{Lex}(\mathcal{C},\textbf{Set})^{op}$, the opposite category of finite-limit-preserving functors into \textbf{Set}. The embedding $\C\rightarrow\Pro(\C)$ is the restricted Yoneda embedding. This follows from \cite[Definition~6.1.1, Remark~6.1.7]{lurie_ultracategories}; see also \cite[Appendix, Corollary 2.8]{alma992977584035101631}.

Alternatively, by \cite[Remark 6.1.6]{lurie_ultracategories}, $\Pro(\C)$ admits the following description:
\begin{itemize}
    \item Objects are cofiltered diagrams of $\C$;
    \item On morphisms, we have the natural bijection 
    \[ \hom_{\Pro(\mathcal{C})}((X_i)_{i\in I}, (Y_j)_{j\in J}) \cong \lim_{j\in J} \colim_{i\in I} \hom_\mathcal{C}(X_i,Y_j). \]
\end{itemize}

\begin{lem}
    The category $\Pro(\defT(T))$ is equivalent to the category of types.
\end{lem}

\begin{proof}
    On objects, given a cofiltered system $(X_i)_{i\in I}$, we can form the a type in the variables $(x_i)_{i\in I}$, where each $x_i$ is pairwise disjoint, by taking the deductive closure of

    \[ \{ X_i(x_i) \ | \ i\in I \} \cup \{ \alpha(x_i)=x_j \ | \ \alpha: X_i \rightarrow X_j \textup{ is a transition map}\}. \]
    Conversely given a type $\{\phi_i(x_i) \ | \ i\in I \}$ in the variables $x=\bigcup x_i$, we define the cofiltered system $X_i = \phi_i(x)$, where there is an inclusion morphism $X_i\rightarrow X_j$ whenever $T\vdash X_i\rightarrow X_j$.

    On morphisms, the equivalence is simply spelling out the natural bijection; for details, see \cite[Ch.~I~\S1.1]{alma992985533896801631}.
\end{proof}

\begin{cor}
    The category of types is a regular category and the embedding $\defT(T)\rightarrow\Pro(\defT(T))$ is a fully faithful regular embedding.
\end{cor}

\begin{proof}
    By \cite[Theorem 1]{BARR1986113}, this is true for any regular category, in particular $\defT(T)$.
\end{proof}

\smallskip

Historically, people studying pro-completions have been interested in \emph{uniform approximation}. The problem is as follows: given a certain diagram of shape $I$ in $\Pro(\C)$, can we find diagrams of the same shape in $\C$, whose cofiltered limit is the given diagram? More explicitly, consider that the category $[I,\Pro(\C)]$ have cofiltered limits and there is a functor $[I,\C]\rightarrow[I,\Pro(\C)]$ by composing with the inclusion functor $\C\rightarrow\Pro(\C)$. The universal property of $\Pro([I,\C])$ then induces a unique functor 
\[E^I_\C: \Pro([I,\C]) \rightarrow [I,\Pro(\C)].\] 
The uniform approximation problem then asks: under what assumptions is $E^I_\C$ an equivalence? This was studied in \cites[Expos\'{e} I Proposition 8.8.5 \addsemicolon]{a.TheorieToposCohomologie1972}[Appendix, Proposition 3.3 \addsemicolon]{alma992977584035101631}[Ch.~II \S3]{meyerphd}. The problem turns out to be quite subtle, and as the most recent contribution \cite[Example 1.5]{henryIndcompletionAccessibilityFunctor2025} points out, there have been mistakes in the past literature. We state the uniform approximation theorem, specialised and dualised to fit our setting.

\begin{thm}[Uniform approximation {\cite[Theorems 1.2 and 1.3]{henryIndcompletionAccessibilityFunctor2025}}]\label{simonhenry}
    Let $I$ be an indexing category.
    \begin{enumerate}
        \item The category $I$ is finite if and only if for any category $\C$ with finite limits, the functor $E^I_\C: \Pro([I,\C]) \rightarrow [I,\Pro(\C)]$ is an equivalence.
        \item The category $I$ is finite and has no non-identity endomorphism if and only if for any category $\C$, the functor $E^I_\C: \Pro([I,\C]) \rightarrow [I,\Pro(\C)]$ is an equivalence.
    \end{enumerate}
\end{thm}

As a corollary, we obtain an easier description of the morphisms of $\Pro(\C)$. Given a morphism $f$ in $\Pro(\C)$, the domain $X$ and the codomain $Y$ can be given by the same indexing category as $(X_i)$ and $(Y_i)$, respectively. Moreover, there are morphisms $(f_i:X_i\rightarrow Y_i)$ in $\C$, whose limits in the arrow category $[\{\bot\rightarrow \top\},\C]$ is $f$.

\begin{defin}
    A morphism $f:X\rightarrow Y$ in $\Pro(\mathcal{C})$ is called a \textit{level morphism} if there is a cofiltered category $I$ with a functor $F:I\rightarrow [\{\bot\rightarrow \top\},\C]$, such that $\lim_I F \cong f$. In this case, we write $f\cong\lim(f_i:X_i\rightarrow Y_i)$, where $f_i = Fi$ for $i\in I$.
\end{defin}

\begin{lem}\label{level morphisms}
    Every morphism in $\Pro(\mathcal{C})$ is a level morphism. 
\end{lem}

\begin{proof}
    Apply part (2) of Theorem \ref{simonhenry} to $I=\{\bot\rightarrow \top\}$.
\end{proof}

\smallskip

Let us now study when two objects of $\Pro(\C)$ are isomorphic. 

\begin{defin}[{\cite[p.~149]{alma992977584035101631}}]
Let $I$ and $J$ be cofiltered categories. We say a functor $F:I\rightarrow J$ is \emph{cofinal} if for every $j\in J$, there is $i\in I$ and a morphism $Fi\rightarrow j$ in $J$.
\end{defin}

\begin{lem}[{\cite[Appendix, Corollary 2.5]{alma992977584035101631}}]\label{cofinal functor}
    Let $F:I\rightarrow J$ be a cofinal functor between cofiltered categories. Then, the pro-objects $(X_j)_{j\in J}$ and $(X_{Fi})_{i\in I}$ are isomorphic via the representing morphisms $(\textup{id}_{X_{Fi}})_{i\in I}$.
\end{lem}

\begin{prop}[{\cite[Proposition~8]{KAMENSKY2007180}}]\label{isomorphism is calculated in monster}
    A morphism in $\Pro(\defT(T))$ is an isomorphism if and only if its interpretation in some/any saturated model is a bijection. In other words, saturated models reflect isomorphisms.
\end{prop} 

\begin{remark}
    Any model $\M: \defT(T)\rightarrow \textbf{Set}$ extends uniquely to a continuous functor $\M: \Pro(\defT(T))\rightarrow \textbf{Set}$. Hence, an isomorphism of $\Pro(\defT(T))$ is mapped to an isomorphism of \textbf{Set}, which is a bijection. The other direction requires a use of compactness on formulas with negation. In particular, the proposition does not hold if $T$ is merely a positive theory. 
\end{remark}

\begin{example}
    We show that non-saturated models might not reflect isomorphisms. Let $\mathcal{L} = \{c_n \ | \ n\in\omega\}$ be a countably infinite set of constants. Let $T$ be the theory which says each $c_n$ is distinct. Then, $p := \{ x \neq c_n \ | \ n\in\omega \}$ is a consistent type. Consider the empty map into $p$.

    Let $\M$ be the prime model, where every element is named by some $c_n$, so that $p$ is not realised. Then, the empty map is a bijection. It is not an isomorphism: using the proposition above, a saturated model realises $p$, so that the empty map is not a bijection.
\end{example}

The following lemma allows us to ignore the differences between categorical limits and set-theoretic intersections.

\begin{lem}\label{level intersections}
    Let $f:X\rightarrow Y$ be a morphism in $\Pro(\defT(T))$. Then, there are formulas $\phi_i,\psi_i$ and definable functions $g_i:\phi_i\rightarrow\psi_i$, such that in any model $\M\models T$, we have $X(\M)=\bigcap \phi_i(\M)$, $Y(\M)=\bigcap \psi_i(\M)$, and $f(\M)=\bigcap g_i(\M)$.
\end{lem}

\begin{proof}
    By Lemma \ref{level morphisms}, assume $f=\lim (f_i:X_i\rightarrow Y_i)$. Moreover, let $\phi_\alpha$ enumerate all the formulas which are consequences of $X$, each considered in the same set of variables as $X$. Likewise let $\psi_\beta$ enumerate the consequences of $Y$. By construction, $X=\bigcap \phi_\alpha$ and $Y=\bigcap \psi_\beta$ in any model.

    Given $i\in I$, the map $\pi_i: X\rightarrow X_i$ implies that $X_i(\pi_i(x))$ is one of the consequences of $X$, so is one of the $\phi_\alpha$. Conversely, given $\alpha$, compactness implies that there is some $i\in I$ such that $X_i(\pi_i(x))\vdash \phi_\alpha(x)$, so that there is a morphism $X_i\rightarrow\phi_\alpha$. This is true for $Y_i$ and $\psi_\beta$ as well. Hence, given $i\in I$, we form the following diagram, which commutes with the projections from $X$ and $Y$.
        \[\begin{tikzcd}
        	{} & \\
        	{\phi_{\alpha_i}:=X_j(\pi_j(x))} \\
        	{X_j} & {Y_j} \\
        	& {\psi_{\beta_i}} \\
        	{X_i} & {Y_i}
        	\arrow["{\pi_j}"', from=2-1, to=3-1]
        	\arrow["{f_j}", from=3-1, to=3-2]
        	\arrow[from=3-1, to=5-1]
        	\arrow[from=3-2, to=4-2]
        	\arrow[from=4-2, to=5-2]
        	\arrow["{f_i}"', from=5-1, to=5-2]
        \end{tikzcd}\]
    Let $g_i$ be the morphism $\phi_{\alpha_i}\rightarrow\psi_{\beta_i}$. Then as $(g_i)$ is cofinal in $(f_i)$, they have the same limit. 
\end{proof}

\medskip

\subsection{Regular epimorphisms}
We already know that $\Pro(\defT(T))$ is regular. We study its regular epimorphisms and obtain a characterisation of saturated models. We first recall the \emph{diagonal fill-in property} of a regular category.

\begin{prop}[{\cite[Proposition 2.4]{Butz1998}}]\label{diagonal fill in}
    In a regular category $\C$, given a commutative square
    \[\begin{tikzcd}
    	X & Y \\
    	{X'} & {Y'}
    	\arrow["f", from=1-1, to=1-2]
    	\arrow["e"', two heads, from=1-1, to=2-1]
    	\arrow["m", hook, from=1-2, to=2-2]
    	\arrow["g"', from=2-1, to=2-2]
    \end{tikzcd}\]
    where $e$ is a regular epimorphism and $m$ is a monomorphism, there is a unique arrow $d:X'\rightarrow Y$ such that the two triangles commute.
\end{prop}

\begin{lem}\label{level regular epimorphism}
    Let $\mathcal{C}$ be regular. Then a morphism in $\Pro(\mathcal{C})$ is a regular epimorphism if and only if it is a cofiltered limit of regular epimorphisms in $\mathcal{C}$.
\end{lem}

\begin{proof}
    $(\Rightarrow)$ \quad The dual statement was proved in \cite[Theorem 1]{BARR1986113}. Let $f=\lim f_i$. Given $i$, we may factorise $f_i$ into a regular epimorphism $e_i$ followed by a monomorphism $m_i$. Denote by $W_i$ the image of this factorisation. By the diagonal fill-in property in Proposition~\ref{diagonal fill in}, we have a unique morphism $Y\rightarrow W_i$, which exhibits $Y$ as the limit of $W_i$.
        \[\begin{tikzcd}
            X & {X_i} & {W_i} \\
            Y && {Y_i}
            \arrow[from=1-1, to=1-2]
            \arrow["f"', two heads, from=1-1, to=2-1]
            \arrow["{e_i}", two heads, from=1-2, to=1-3]
            \arrow["{m_i}", hook, from=1-3, to=2-3]
            \arrow[dashed, from=2-1, to=1-3]
            \arrow[from=2-1, to=2-3]
        \end{tikzcd}\]
    Then, $f$ is the cofiltered limit of $e_i$. This shows that all regular epimorphisms in $\Pro(\C)$ are cofiltered limits of regular epimorphisms in $\C$.

\smallskip

    $(\Leftarrow)$ \quad  In \textbf{Set}, filtered colimits commute with finite limits. Since limits and colimits of presheaves are calculated pointwise, this holds also for $\textup{Lex}(\mathcal{C},\textbf{Set}) \simeq \Pro(\mathcal{C})^{op}$. Dualising, in $\Pro(\mathcal{C})$ cofiltered limits commute with finite colimits. In particular, cofiltered limits of regular epimorphisms are regular epimorphisms. 
\end{proof}

Warning: it is not true that if $(f_i)$ represents a regular epimorphism $f$, then each $f_i$ is a regular epimorphism. Instead, the lemma says that we can find some representing $(e_i)$, each of which is a regular epimorphism.

\smallskip

The above proof also gives a description of the regular epimorphism-monomorphism factorisation in $\Pro(\mathcal{C})$. Let $f=\lim f_i$. Then, one factorises each $f_i = m_i \circ e_i$ through its image $W_i$. Then, using the diagonal fill-in property, whenever we have $j\rightarrow i$, there is an induced $W_j \rightarrow W_i$. Hence, we can take the limit $\lim W_i$, which gives precisely the image of $f$ with the regular epimorphism $\lim e_i$ and the monomorphism $\lim m_i$. This was first shown in \cite{BARR1986113}. For a more modern account, see \cite[Theorem 12]{kanalas2025puremapsstrictmonomorphisms}.
\[\begin{tikzcd}
	X & {\lim W_i} & Y \\
	{X_j} & {W_j} & {Y_j} \\
	{X_i} & {W_i} & {Y_i}
	\arrow["{\lim e_i}"', dashed, two heads, from=1-1, to=1-2]
	\arrow["f", curve={height=-12pt}, from=1-1, to=1-3]
	\arrow[from=1-1, to=2-1]
	\arrow["{\lim m_i}"', dashed, hook, from=1-2, to=1-3]
	\arrow[from=1-2, to=2-2]
	\arrow[from=1-3, to=2-3]
	\arrow["{e_j}"', two heads, from=2-1, to=2-2]
	\arrow[from=2-1, to=3-1]
	\arrow["{m_j}"', hook, from=2-2, to=2-3]
	\arrow[dashed, from=2-2, to=3-2]
	\arrow[from=2-3, to=3-3]
	\arrow["{e_i}"', two heads, from=3-1, to=3-2]
	\arrow["{m_i}"', hook, from=3-2, to=3-3]
\end{tikzcd}\]
With $\mathcal{C} = \defT(T)$, the image of a morphism $f:X\rightarrow Y$ is given by the type 
\[Y \cup \{\exists x \ (X_{i_j}(x) \wedge f_j(x)=y ) \ | \ j\in J\}.\]

\smallskip

We now recall the notion of $\lambda$-regularity.

\begin{defin}[{\cite[Definition 2.2]{MAKKAI1990225}}]
    Let $\lambda$ be a regular cardinal. 
\begin{enumerate}
    \item A category $\C$ is $\lambda$-\textit{regular} (resp. $\infty$-\emph{regular}) if it is regular, has $\lambda$-limits (resp. small limits), and regulars epimorphisms are preserved by $\lambda$-products (resp. small products).
    \item  A category $\C$ is $\lambda$-\textit{exact} (resp. $\infty$-\emph{exact}) if it is $\lambda$-regular (resp. $\infty$-regular) and is exact.
    \item  A functor between $\lambda$-regular categories is $\lambda$-\emph{regular} (resp. $\infty$-\emph{regular}) if it preserves regular epimorphisms and $\lambda$-limits (resp. small limits)
\end{enumerate}
\end{defin}

\begin{thm}\label{pro is infty regular}
    Let $\C$ be regular. Then $\Pro(\C)$ is $\infty$-regular.
\end{thm}

\begin{proof}
    We already know that $\Pro(\C)$ is regular and complete. Moreover, a cofiltered limit of regular epimorphisms is a regular epimorphism. Since $\C$ has finite limits, every small product in $\Pro(\C)$ is a cofiltered limit of representibles. Hence, a product of regular epimorphisms is a cofiltered limit of regular epimorphisms, which is itself a regular epimorphism.

    More explicitly, let $f_i: X_i \rightarrow Y_i$ be regular epimorphisms for each $i\in I$, where $I$ is an infinite but small set. Let $J$ be the poset category of non-empty finite subsets of $I$. Then $J^{op}$ is a cofiltered category. For each $j=\{i_1,\ldots,i_n\}\in J$, let $g_j = \prod_{k=1}^{n} f_{i_k}$. Then, $\lim_{j\in J^{op}} g_j = \prod_{i\in I} f_i$.
\end{proof}

\begin{remark}
    The Axiom of Choice is equivalent to the statement that, in \textbf{Set}, infinite products of regular epimorphisms are regular epimorphisms. Theorem~\ref{pro is infty regular} says that the pro-completion of any regular category satisfies the Axiom of Choice in this sense.
\end{remark}

\medskip

\subsection{Saturation}\label{subsect_saturation}

We take a brief detour to obtain a characterisation of saturation. We first give a motivating example.

\begin{example}
    Recall that an object $X$ is \textit{inhabited} if the unique morphism $X\rightarrow \top$ into the terminal object is a regular epimorphism. In $\defT(T)$, an object is not inhabited if and only if it is $\bot$. We prove this also holds in $\Pro(\defT(T))$. It is clear that the morphism $\bot\rightarrow \top$ is not a regular epimorphism in $\Pro(\defT(T))$. Conversely, suppose $f:X\rightarrow \top$ is not a regular epimorphism in $\Pro(\defT(T))$. Note that by Lemma~\ref{functions of Pro}, we have $X\cong\bot$ if and only if there is a morphism $X\rightarrow\bot$. Now, $f$ is represented by some $f_i:X_i\rightarrow \top$. If all the $f_i$ are regular epimorphisms, then $f$ is also a regular epimorphism by Lemma~\ref{level regular epimorphism}. Hence, there exists some $i_0$ such that $f_{i_0}$ is not a regular epimorphism, so in fact $X_{i_0} \cong \bot$. There is a projection map $X\rightarrow X_{i_0}$, so in fact $X\cong\bot$.

    In $\defT(T)$, an inhabited object is always non-empty in any model, but in $\Pro(\defT(T))$, an inhabited object can be empty in a model by the omitting type theorem. Categorically, a model $\M$ is a coherent functor $\defT(T)\rightarrow\textbf{Set}$; in particular $\M$ preserves regular epimorphisms in $\defT(T)$. Now, $\M$ admits a unique extension to a continuous functor $\Pro(\defT(T))\rightarrow \textbf{Set}$, but this extension might not be preserve regular epimorphisms in $\Pro(\defT(T))$. 
\end{example}

A $\kappa$-saturated model is a model which realises all types with sets of parameters strictly smaller than $\kappa$. Accordingly, we need a way of measuring the size of a set of parameters. 

\begin{defin}
    The \textit{length} of an object $X\in\Pro(\C)$, denoted by $l(X)$, is the least cardinality $|I|$ such that there is a monomorphism $X\hookrightarrow \prod_{i\in I}Y_i$, where $Y_i\in\C$.

    The length of a morphism in $\Pro(\C)$ is the length of its codomain.
\end{defin}

\begin{lem}
    Suppose $X\in\Pro(\defT(T))$ has infinite length. Then $X$ is isomorphic to a type $p(x)$, where $l(X)=|x|$.
\end{lem}

\begin{proof}
    Let $\lambda=l(X)$, so that there is a monomorphism $m:X\hookrightarrow\prod_{\alpha<\lambda} Y_\alpha$, where $Y_\alpha\in\defT(T)$ is a formula $\psi_\alpha(y_\alpha)$. Let $p(x)$ be the image of the regular epimorphism-monomorphism factorisation of $m$. Since $m$ is monomorphism, $X$ is isomorphic to $p(x)$. Moreover, the image is in the same tuple of variable as $\prod_{\alpha<\lambda} Y_\alpha$, which is $(y_\alpha)_{\alpha<\lambda}$. This has length $\lambda$. 
\end{proof}

\begin{remark}
    An object with finite length in fact has length 1. Not every type in a finite tuple of variables can be expressed as a type in a single variable. However, this is not a big problem: every type in a finite tuple of variables can be expressed as a type in a single variable in $T^{eq}$.
\end{remark}

\begin{remark}\label{cocompact}
    Another way of measuring the size of an object $X$ in a category (for example $\Pro(\C)$) is by considering whether the functor $\Hom(-,X)$ preserves $\lambda$-cofiltered limits. If so, then $X$ is called $\lambda$-\emph{cocompact}. The \emph{size} of $X$ is the least $\lambda$ such that $X$ is $\lambda$-cocompact. This roughly corresponds to the cardinality of the smallest cofiltered diagram landing in $\C$, whose limit in $\Pro(\C)$ is $X$. We say an object is \emph{cocompact} if it is $\aleph_0$-cocompact. By \cite[Theorem 7.2]{DanielC2002}, an object is cocompact if and only if it is isomorphic to an object in $\C$. We write $\textup{Cocompact}(\C)$ for the full subcategory of $\C$ consisting of the cocompact objects. Then, we can recover $\C$ from $\Pro(\C)$ uniquely via $\textup{Cocompact}(\Pro(\C)) \simeq \C$. Specialising to $\defT(T)$, an object is cocompact if and only if it is an isolated type.
    
    The size matches with the length when they are larger than $|T|$. However, when they are smaller, they can be different. For example, a non-isolated type in a finite tuple has length 1, but infinite size.
\end{remark}

\smallskip

We show that saturation is equivalent to preservation of regular epimorphisms in $\Pro(\defT(T))$.

\begin{thm}
    A model $\M$ is $\lambda$-saturated, if and only if the functor $\M:\Pro(\defT(T))\rightarrow \textbf{Set}$ maps regular epimorphisms with length strictly less than $\lambda$ to surjections. 
\end{thm}

\begin{proof}
    $(\Rightarrow)$ \quad  Let $f=\lim (f_i: X_i\rightarrow Y_i)$ be a morphism $X\rightarrow Y$ in $\Pro(\defT(T))$, where $l(Y)<\lambda$. By Lemma \ref{level regular epimorphism}, assume all $f_i$ are regular epimorphisms. We write $p_i$ for the projection $X\rightarrow X_i$ and $\pi_i$ for the projection $Y\rightarrow Y_i$.

    Let $b\in Y(\M)$ realise $Y$ in $\M$, with $|b|\leq l(Y)<\lambda$. Then, the type $X\cup\{f_i(p_i(x))=\pi_i(b) \ |\ i\in I \}$ is consistent, since each $f_j$ is a surjection. Since $\M$ is $\lambda$-saturated, the type is realised by some $a\in X(\M)$. It is clear that $f^\M(a)=b$.

\smallskip

    $(\Leftarrow)$ \quad Let $p(x)\in S(A)$ be a type over some parameter set $A\subset \M$, where $|A|<\lambda$. Let $X(x,y_A)$ be the type $p'\cup \tp(A)$, where $\tp(A)$ is in the variables $y_A$ and $p'$ is obtained from $p$ by replacing $A$ with variables $y_A$. Then, that $\M$ realises $p(x)$ will follow once we show the surjectivity of the projection $\pi_A: X\rightarrow \tp(A)$. Note that $l(\pi_A) = l(\tp(A)) \leq |A| < \lambda$. It suffices to show that $\pi_A$ is a cofiltered limit of regular epimorphisms.

    The map $\pi_A$ is a cofiltered limit of the maps $[\phi(x,y_A) \wedge \psi(y_A)] \rightarrow [\psi(y_A)]$ for each $\phi\in p'$ and $\psi\in\tp(A)$. By the completeness of $T$, we have
    \[T \vdash \forall y_A \ (\psi(y_A) \rightarrow \exists x \ \phi(x,y_A)).\]
    Hence, these maps are indeed regular epimorphisms.
\end{proof}

\begin{remark}
    A coherent functor from $\Pro(\defT(T))$ to \textbf{Set} is therefore a $\lambda$-saturated model for all small $\lambda$. Monster models are precisely the homogeneous coherent functors. 

    Another way of characterising saturated models categorically is via a lifting property. $\M$ is $\lambda$-saturated if and only if for all models $\N\rightarrow \N'$, where $|\N'|< \lambda$, and elementary embedding $\N\rightarrow \M$, there is a lift $\N'\rightarrow \M$.
    \[\begin{tikzcd}
        {\N'} & \M \\
        \N
        \arrow[dashed, from=1-1, to=1-2]
        \arrow[from=2-1, to=1-1]
        \arrow[from=2-1, to=1-2]
    \end{tikzcd}\]
\end{remark}

We give an example to illustrate why saturated models are relevant.

\begin{example}[Infinitary Chinese remainder theorem]
    Let $T$ be the theory of $(\mathbb{Z},0,1,+,\times)$. For every prime $p$, the equivalence relation $\textup{mod } p$, as well as its quotient $\mathbb{Z}_p$ and quotient map $\pi_p$ is definable. Let $\pi = \langle \pi_p \rangle_p :\M\rightarrow \prod_p \mathbb{Z}_p$. This map is a cofiltered limit of regular epimorphisms $\pi_p$, so is a regular epimorphism. However, the interpretation in the standard model is not a surjection; for example, the tuple $(1,1,\ldots)$ does not have a pre-image. In a non-standard model $\M$, the tuple does have a pre-image whenever $\M$ realises the type $\{ x = 1 \textup{ mod } p\ |\ p \textup{ prime}\}$.
\end{example}

%% file: 4elimination_of_hyperimaginaries.tex
\bigskip

\section{Hyperimaginaries}\label{hyperimaginaries}
In this section, we review some basics about hyperimaginaries. As we said in the Introduction, hyperimaginaries are important for the study of simple unstable theories. For standard textbook references, see \cites[\S3\addsemicolon]{alma992983361725701631}[\S\S15, 18]{casanovas2011simple}. 

\begin{defin}
    A theory $T$ \textit{eliminates hyperimaginaries} if for every type-definable equivalence relation $E$ on a type $X$, there is a sequence of definable equivalence relations $E_i$ on definable sets $X_i$, such that $\bigcap_i E_i \upharpoonright_X = E$.
\end{defin}

The following two propositions show that instead of working with $T$, we may work over $T^{eq}$.

\begin{prop}\label{EHI invariant under eq}
    A theory $T$ eliminates hyperimaginaries if and only if $T^{eq}$ also eliminates hyperimaginaries.
\end{prop}

\begin{proof}
    $(\Rightarrow)$ \quad  Let $E((y_i)_{i\in I},(y'_i)_{i\in I})$ be a type-definable equivalence relation in $T^{eq}$, where $y_i$ is of sort $\M/R_i$. To simplify notation, we write $y_I$ for $(y_i)_{i\in I}$. Then, $E(\pi_I(x_I),\pi_I(x'_I))$ is a type-definable equivalence relation in $T$, where $\pi_I$ is the product of the quotient maps $\pi_i: \M\rightarrow \M/R_i$. Since $T$ eliminates hyperimaginaries, there are definable equivalence relations $E_J(x_J,x'_J)$ in $T$, where $J\in\mathcal{J}$ is a finite subset of $I$, such that for all $a_I,a'_I\in\M$,
    \begin{equation}
        E(\pi_I(a_I),\pi_I(a'_I)) \ \Leftrightarrow \ E_J(a_J,a'_J) \textup{ for all }J\in\mathcal{J}. \tag{$*$}
    \end{equation}  
    For each $J$, consider the definable relation
    \[ \varepsilon_J(y_J,y'_J) := \exists x_J, x'_J \ (\pi_J(x_J)=y_J \wedge \pi_J(x'_J)=y'_J \wedge E_J(x_J, x'_J)). \]
    The reflexivity and symmetry of $\varepsilon_J$ is easy to see. 
    For transitivity, we need the following.

    \begin{claim}\label{EJ RJ}
        Suppose $\pi_J(a_J)=\pi_J(a'_J)$ for some $J\in\mathcal{J}$. Then, $E_J(a_J,a'_J)$ holds.
    \end{claim}

    \begin{proof}
        Suppose $\pi_J(a_J)=\pi_J(a'_J)$ holds. Extend $a_J$ in any way to a tuple $a_I$ and extend $a'_J$ to $a'_I$ such that $a_i = a'_i$ for all $i\not\in J$. Then, we have $\pi_I(a_I)=\pi_I(a'_I)$, so that by the reflexivity of $E$, we have $E(\pi_I(a_I),\pi_I(a'_I))$. By $(\ast)$, we have $E_J(a_J,a'_J)$. 
    \end{proof}
    
    Back to the proof of the transitivity of $\varepsilon_J$. Suppose we have $\varepsilon_J(b_J,b'_J)$ and $\varepsilon_J(b'_J,b''_J)$. By definition, there are $a_J,a'_J,\alpha'_J,$ and $\alpha''_J$ such that 
    \begin{align*}
        \pi_J(a_J) &=b_J &\wedge& &\pi_J(a'_J) &=b'_J &\wedge & &E_J(a_J,a'_J); \\
        \pi_J(\alpha'_J) &= b'_J & \wedge & &\pi_J(\alpha''_J) &=b''_J &\wedge &&E_J(\alpha'_J,\alpha''_J).
    \end{align*}
    By the Claim, we have $E_J(a'_J,\alpha'_J)$, so the transitivity of $E_J$ implies that $E_J(a_J,\alpha''_J)$. Then, $a_J,\alpha''_J$ witnesses $\varepsilon_J(b_J,b''_J)$ as required.
    
    Finally, we prove that $\bigcap_{J\in\mathcal{J}} \varepsilon_J = E$. Fix any $b_i,b'_i$ of sort $\M/R_i$, and any $a_i,a'_i$ such that $\pi_i(a_i)=b_i$ and $\pi_i(a'_i)=b'_i$ for all $i\in I$. Note that if we have $\varepsilon_J(b_J,b'_J)$, so that there are some $\alpha_J,\alpha'_J$ such that $\pi_J(\alpha_J)=b_J$ and $\pi_J(\alpha'_J)=b'_J$ satisfying $E_J(\alpha_J,\alpha'_J)$, by the Claim, we also have $E_J(a_J,a'_J)$. Hence, we conclude that
    \[\varepsilon_J(b_J,b'_J) \textup{ for all } J\in\mathcal{J} \ \Leftrightarrow \ E_J(a_J,a'_J) \textup{ for all } J\in\mathcal{J} \ \Leftrightarrow \ E(b_I,b'_I), \]
    where the second equivalence comes from $(*)$.

\smallskip

    $(\Leftarrow)$ \quad  Suppose $E$ is a type-definable equivalence relation in $T$. Then, there are definable equivalence relations $E_i$ in $T^{eq}$ such that $\bigcap E_i = E$. Now, $E_i$ is a definable relation on the home sort, so in fact each $E_i$ is definable in $T$. 
\end{proof}

\begin{prop}\label{if prodef exact then def exact}
    If $\Pro(\defT(T))$ is exact, then $\defT(T)$ is exact.
\end{prop}

\begin{proof}
    Suppose $\Pro(\defT(T))$ is exact. Fix an equivalence relation $E\rightrightarrows X$ in $\defT(T)$. Then, there is a quotient $f: X\rightarrow Y$ in $\Pro(\defT(T))$, represented by some $f_i: X\rightarrow Y_i$ in $\defT(T)$. Logically, $f$ being the quotient map for $E$ means 
    \[ E(x,x') \ \Leftrightarrow \ f_i(x)=f_i(x') \textup{ for all }i. \]
    By compactness, there exists an $i_0$ such that $E(x,x')$ is equivalent to $f_{i_0}(x)=f_{i_0}(x')$. Then, $f_{i_0}:X\rightarrow Y_{i_0}$ is the quotient for $E$ in $\defT(T)$.
\end{proof}

\begin{thm}\label{eliminate hyper iff pro exact}
    A theory $T$ eliminates hyperimaginaries if and only if $\Pro(\defT(T^{eq}))$ is exact.
\end{thm}

\begin{proof}
    By Proposition \ref{EHI invariant under eq}, we assume $T$ eliminates imaginaries.

\smallskip

    $(\Rightarrow)$ \quad Let $E\rightrightarrows X$ be an equivalence relation in $\Pro(\defT(T))$. Let $(E_i\rightrightarrows X_i)_{i\in I}$ be a sequence of definable equivalence relations which eliminate $E$. Without loss of generality, assume $I$ is a cofiltered system. As $T$ eliminates imaginaries, the quotients $X_i/E_i$ are definable. Given a map $i\rightarrow j$ in $I$,  the universality of $X_i/E_i$ as a coequaliser induces a map $X_i/E_i\rightarrow X_j/E_j$. 
\[\begin{tikzcd}
	{E_i} & {X_i} & {X_i/E_i} \\
	{E_j} & {X_j} & {X_j/E_j}
	\arrow[shift right, from=1-1, to=1-2]
	\arrow[shift left, from=1-1, to=1-2]
	\arrow[from=1-1, to=2-1]
	\arrow["{e_i}", from=1-2, to=1-3]
	\arrow[from=1-2, to=2-2]
	\arrow[dashed, from=1-3, to=2-3]
	\arrow[shift right, from=2-1, to=2-2]
	\arrow[shift left, from=2-1, to=2-2]
	\arrow["{e_j}"', from=2-2, to=2-3]
\end{tikzcd}\]

    Hence, we obtain the cofiltered limit $\lim X_i/E_i \in \Pro(\defT(T))$. Moreover, the maps $e_i:X_i \rightarrow X_i/E_i$ induce a map $e=\lim e_i:X\rightarrow \lim X_i/E_i$. By Lemma \ref{level regular epimorphism}, $e$ is a regular epimorphism. Kernel pairs commute with cofiltered limits, so we have
    \[ X\times_{\lim X_i/E_i} X 
    \cong \lim (X_i\times_{X_i/E_i}X_i) 
    \cong \lim E_i
    \cong E.\]
    Hence, $E\rightrightarrows X$ is a kernel pair with coequaliser $\lim X_i/E_i$ as required.

    \smallskip

    $(\Leftarrow)$ \quad Let $E\rightrightarrows X$ be a type-definable equivalence relation. By exactness, there is a type-definable quotient $Y = X/E$. By Lemma \ref{level intersections}, the map $\pi:X\rightarrow X/E$ is the intersection of some $(\pi_i:X_i\rightarrow Y_i)$, where $X=\bigcap X_i$ and $Y=\bigcap Y_i$.

    Let $E_i$ be the pullback $X_i \times_{Y_i} X_i$. It is clear that $E_i$ is a definable equivalence relation. Moreover, in a saturated model $\M$ for every $a,b\in \M(X)$, we have:
    \begin{align*}
        E(a,b) &\Leftrightarrow \pi(a)=\pi(b) \\
        &\Leftrightarrow \pi_i(a)=\pi_i(b)  \textup{ for all } i \\
        &\Leftrightarrow E_i(a,b) \textup{ for all } i.
    \end{align*}
    Hence, $\M(E)=\bigcap \M(E_i)$. Since $T$ is complete, $E=\bigcap E_i$.
\end{proof}

Noting that the proof of Theorem \ref{eliminate hyper iff pro exact} does not use the fact that $\defT(T)$ is boolean, we in fact obtain a more general theorem: the pro-completion of an exact category $\C$ is exact if and only if every equivalence relation in $\Pro(\C)$ is the cofiltered limit of equivalence relations in $\C$. The latter condition was studied in \cite[Ch.~III \S4]{meyerphd}. Specifically, it is proved that reflexive symmetric relations in $\Pro(\C)$ are cofiltered limits of reflexive symmetric relations in $\C$. This is in fact a corollary of uniform approximation. In the notation of Theorem \ref{simonhenry}, let $I$ be the following category
    \[\begin{tikzcd}
    	R & {X,}
    	\arrow["s"{description}, from=1-1, to=1-1, loop, in=55, out=125, distance=10mm]
    	\arrow["{p_1}"{description}, shift left=3, from=1-1, to=1-2]
    	\arrow["{p_2}"{description}, shift right=3, from=1-1, to=1-2]
    	\arrow["r"{description}, from=1-2, to=1-1]
    \end{tikzcd}\]
where $r$ satisfies the usual property of reflexivity and $s$ satisfies the usual property of symmetry. Now, $I$ has a non-identity endomorphism $s$, but part (1) of Theorem \ref{simonhenry} still applies as $\defT(T)$ has finite limits.

However, one cannot do this for transitivity. Indeed, the diagram for an equivalence relation is the following:
    \[\begin{tikzcd}
    	{R\times_X R} & R & {X,}
    	\arrow["t"{description}, from=1-1, to=1-2]
    	\arrow["s"{description}, from=1-2, to=1-2, loop, in=55, out=125, distance=10mm]
    	\arrow["{p_1}"{description}, shift left=3, from=1-2, to=1-3]
    	\arrow["{p_2}"{description}, shift right=3, from=1-2, to=1-3]
    	\arrow["r"{description}, from=1-3, to=1-2]
    \end{tikzcd}\]
where $r$ and $s$ are as before, and $t$ satisfies the usual property of transitivity. Of course, one can let $I$ to be the category above and apply Theorem \ref{simonhenry}. Then, one obtains morphisms $t_i:W_i\rightarrow E_i$, whose cofiltered limit is $t$, but there is no guarantee that $W_i \cong E_i\times_{X_i} E_i$, or that the relations $E_i\rightrightarrows X_i$ are equivalence relations.

\smallskip

The proof of Theorem \ref{eliminate hyper iff pro exact} provides the precise construction of quotients in $\Pro(\defT(T))$. We apply this to the example of the infinitary Chinese remainder theorem with $T=\textup{Th}(\mathbb{Z},+,\times,0,1)$. We are interested in the quotient for the type-definable equivalence relation $\bigcap_{p \textup{ prime}} \textup{mod }p$. While $T$ might not eliminate all hyperimaginaries, $T$ certainly eliminates this particular one, which is already given as the intersection of definable equivalence relations. Hence, the quotient is also type-definable.

The quotient is constructed as the cofiltered limit of the quotients of the equivalence relations $\bigcap_{p\in I} \textup{mod } p$, where $I$ is a finite set of prime numbers.
The ordinary Chinese remainder theorem states that finite products and quotients commute: $\prod_{p\in I}\mathbb{Z}/p\mathbb{Z} \cong \mathbb{Z}/(\bigcap_{p\in I}\textup{mod }p)$. Hence, the quotient of $\bigcap_{p \textup{ prime}} \textup{mod }p$ is the profinite integers $\prod_{p \textup{ prime}}\mathbb{Z}/p\mathbb{Z}$.

This might seem surprising. Consider that $\bigcap_{p \textup{ prime}} \textup{mod }p$ is the trivial equivalence relation on $\mathbb{Z}$, with quotient $\mathbb{Z}$ itself, not $\prod_{p \textup{ prime}}\mathbb{Z}/p\mathbb{Z}$. The problem is that one should not calculate quotients inside the standard model $\mathbb{Z}$. Indeed, given a sufficiently saturated model $\M$, we do have $\M/(\bigcap_{p \textup{ prime}} \textup{mod }p) \cong \prod_{p \textup{ prime}} \mathbb{Z}/p\mathbb{Z}$. 

\begin{example}
    It was noted in \cite[\S1]{Pillay_Poizat_1987} that the theory \textsf{RCF} of real closed fields does not eliminate hyperimaginaries. In particular, the type-definable equivalence relation 
    \[ E(x,y):= \{|x-y| < \frac{1}{n} \ | \ n\in\mathbb{Z}^+ \}\]
    of being infinitesimally close together is not eliminated. The original argument uses o-minimality, but Theorem \ref{eliminate hyper iff pro exact} yields a new argument. Noting that \textsf{RCF} eliminates imaginaries, suppose \textsf{RCF} also eliminates hyperimaginaries, so that $\Pro(\defT(\textsf{RCF}))$ is exact. Then, the quotient of $E$ is type-definable, say by a type $X$. Now, an $E$-hyperimaginary is a set of numbers infinitesimally close to each other and is represented by their shared standard part. Hence, $X$ interpreted by a saturated model $\M$ is isomorphic to the standard $\mathbb{R}$ and does not contain an infinitesimal. However, one can force infinitesimals to exist in a saturated model by realising the type $X\cup\{0<x<\frac{1}{n}\ | \ n\in\mathbb{Z}^+ \}$, yielding a contradiction.
\end{example}

We draw a corollary from Theorem~\ref{eliminate hyper iff pro exact}.

\begin{cor}
    A theory $T$ eliminates imaginaries and hyperimaginaries if and only if $\Pro(\defT(T))$ is exact.
\end{cor}

\begin{proof}
    $(\Rightarrow)$ \quad By Proposition~\ref{EI iff exact} and Theorems~\ref{Teq is exact} and ~\ref{eliminate hyper iff pro exact}. 
    
    \smallskip
    
    $(\Leftarrow)$ \quad By Proposition~\ref{if prodef exact then def exact}, we know $\defT(T)$ is exact, so by Proposition~\ref{EI iff exact}, $T$ eliminates imaginaries. By Theorem~\ref{Teq is exact}, we know $\defT(T) \simeq \defT(T)_{ex/reg} \simeq \defT(T^{eq})$. Hence, we have $\Pro(\defT(T^{eq})) \simeq \Pro(\defT(T))$, which is exact, so by Theorem~\ref{eliminate hyper iff pro exact}, $T$ eliminates hyperimaginaries.
    
\end{proof}

%% file: 5heq.tex
\bigskip

\section{The heq construction and the exact completion}\label{heq}
We have now established in Theorem \ref{eliminate hyper iff pro exact} a correspondence between two properties---elimination of hyperimaginaries and the exactness of $\Pro(\defT(T)_{ex/reg})$. Just as Theorem \ref{Teq is exact} is a constructive version of Proposition \ref{EI iff exact}, in this section we show there is a constructive version of Theorem \ref{eliminate hyper iff pro exact}. Indeed, we will show that the heq construction corresponds to the ex/reg completion of $\Pro(\defT(T))$.

\medskip

\subsection{The heq construction}
We recall the heq construction from \cite{DobrowolskiJan2022Kipl}. 

As notation, given a type-definable equivalence relation $E(x,x')$, we denote by $[x]_E$ the $E$-hyperimaginary corresponding to the $E$-class of $x$. When it is clear, we omit the subscript and write $[x]$ for $[x]_E$. If $y$ is variable of the sort of an $E$-hyperimaginary, we denote by $y_r$ a variable of the home sort corresponding to a representative. Standard model theoretic notation dictates that one can write $a\in\M$ when $a$ is a finite tuple in $\M$; we shall extend this abuse of notation to the case when $a$ is a small infinite tuple.

\begin{defin}[{\cite[Definitions 10.7 and 10.8]{DobrowolskiJan2022Kipl}}]
    Let $\M$ be a model of $T$. Let $\E$ be a set of type-definable equivalence relations.
    \begin{enumerate}
        \item The language $\Lang_\E$ expands the original language $\Lang$ by sorts $S_E$ for every $E\in\E$. For every $\Lang$-formula $\phi(x,y)$ and type-definable equivalence relatio $E$ on the sort of $y$, we add in a new predicate $R_\phi(x,[y]_E)$. However, we consider $\Lang_\E$ as a language in positive logic, even though $T$ is Boolean. 
        \item The $\Lang_\E$-structure $\M^\E$ expands $\M$ as follows: the sort $S_E$ is interpreted as the quotient of $E(\M)$. For each $\Lang$-formula $\phi(x,y)$, type-definable equivalence relation $E$ on the sort of $y$, and elements $a,b\in\M$,
        \[ \M^\E \models R_\phi(a,[b]_E) \ \Leftrightarrow \ \textup{there is }b' \textup{ such that } \M\models \phi(a,b')\wedge E(b,b'). \]
    \end{enumerate}
\end{defin}

From now on, fix a monster model $\mathcal{U}$ of $T$.

One would like to add in all the hyperimaginaries, but this is a proper class since the variables can be arbitrarily long. Luckily, we have the following lemma. 

\begin{lem}[{\cites[Corollary 3.3\addsemicolon]{Ben-Yaacov2003c}[Fact 1.1]{Hart2000-HARCAC-24}}]
    Let $E(x,y)$ be a type-definable equivalence relation. Then there is a set $I$ and type-definable equivalence relations $E_i(x_i,y_i) : i\in I$, where $x_i$ is a subtuple of $x$ with $|x_i|\leq |T|$, such that for all $a,b\in\mathcal{U}$,
    \[ \mathcal{U}\models E(a,b) \ \Leftrightarrow \ \mathcal{U}\models E_i(a_i,b_i) \textup{ for all }i\in I. \]
    Moreover, if $T$ is Boolean, we may assume $|x_i|$ is countable.
\end{lem}

\begin{defin}
    We define $heq$ to be the set of type-definable equivalence relations $E(x,y)$, where $|x|\leq|T|$.
\end{defin}

The following is the key lemma that allows us to describe $\mathcal{U}^{heq}$ using $\mathcal{U}$.

\begin{lem}[{\cite[Lemma 10.10]{DobrowolskiJan2022Kipl}}]\label{heq translate}
    Let $\phi(x)$ be an $\Lang_{heq}$-formula. Then there is a set $\Sigma(x_r)$ of $\Lang$-formulas such that for all $a\in\mathcal{U}^{heq}$ and $a_r\in\mathcal{U}$ representing $a$,
    \[ \mathcal{U}^{heq}\models\phi(a) \ \Leftrightarrow \ \mathcal{U}\models\Sigma(a_r). \]
\end{lem}

\begin{remark}[{\cite[Lemma 10.11]{DobrowolskiJan2022Kipl}}]
    It is clear that the lemma works even when we replace the $\Lang_{heq}$-formula with a set of $\Lang_{heq}$-formulas.
\end{remark}

Shelah's original eq construction specifically expands a structure with imaginaries, so that this expansion eliminates imaginaries. This is not the case with heq and hyperimaginaries. One does not apply the heq construction, so that the expanded structure eliminates hyperimaginaries. Rather, the heq construction simply expands a structure \emph{with} hyperimaginares.

\begin{remark}[{\cite[Lemma 10.13]{DobrowolskiJan2022Kipl}}]\label{typedef quotient heq}
    Let $E(x_r,y_r)$ be a type-definable equivalence relation in $T$ with $|x_r|=I$. Then, the quotient map $\mathcal{U}^I \rightarrow S_E$ is type-definable in $\mathcal{U}^{heq}$. This is given as follows: for every formula $\phi(x_r,y_r) \in E(x_r,y_r)$, we obtain a predicate $R_\phi(x_r,y)$. We claim that the quotient map is precisely the realisations of the type $P:= \{R_\phi(x_r,y)\ |\ \phi\in E\}$.

    To see this, assume $\mathcal{U}^{heq}\models R_\phi(a,[a'])$ for all $\phi\in E$, for some $a,a'\in\mathcal{U}^I$. We wish to show that $[a']$ is the representative of $a$. Consider the type
    \[ \Gamma(y_r) = E(a,y_r) \cup E(y_r,a'). \]
    We show $\Gamma$ is finitely satisfiable. Given $\phi\in E$, we have $\mathcal{U}^{heq}\models R_\phi(a,[a'])$, so there is some $a''\in\mathcal{U}$ such that $\mathcal{U}\models \phi(a,a'')\wedge E(a'',a')$. By compactness and saturation, there is $b\in\mathcal{U}$ realising $\Gamma$. Then, $[a]=[b]=[a']$ as required. This shows that the realisations of $P$ are contained in the quotient map.

    For the converse, assume $\mathcal{U}\models E(a,a')$. In particular, $\mathcal{U}^{heq}\models R_\phi(a,[a'])$ for all $\phi\in E$.
\end{remark}

It was shown in \cite[Example 3.1.6]{alma992983361725701631} that $\textup{Th}(\mathcal{U}^{heq})$ does not axiomatise the class $\{ \N^{heq} \ | \ \N\models T\}$. In particular, the type $P$ in Remark \ref{typedef quotient heq} might not define a function (let alone the quotient map) in some model $\N\models\textup{Th}(\mathcal{U}^{heq})$. Otherwise, $P$ would be a morphism in $\Pro(\defT(\textup{Th}(\mathcal{U}^{heq})))$ and thus a cofiltered limit of functions definable in $\textup{Th}(\mathcal{U}^{heq})$. In other words, we would require there to be a subset $E_0\subseteq E$ such that $\{R_\phi\ | \ \phi\in E_0\}$ defines the same set as $P$ and for every $\phi\in E_0$, the predicate $R_\phi$ defines a function. The following example shows this does not happen.

\begin{example}
    Let $T$ be the theory of infinite sets in the empty language. Let $E$ be the discrete equivalence relation on tuples $(x_n)_{n\in\omega}$: 
    \[E((x_n),(y_n)) \ \Leftrightarrow \ x_n=y_n \textup{ for all }n.\]
    Then, the sort $S_E$ is interpreted in $\U^{heq}$ by $\U^\omega$. A typical formula $\phi(x_1,\ldots,x_m,(y_n)_{n\in\omega})\in E$ has the form $\bigwedge_{i=1}^m x_i=y_i$. The interpretation of $R_\phi(x_1,\ldots,x_m,y)$ in $\U^{heq}$ is that $(x_1,\ldots,x_m)$ is a subtuple of $y$ (which is an element of $\U^\omega$). In particular, given $a_1,\ldots,a_m$, there are infinitely many $b$ such that $R_\phi(a_1,\ldots,a_m,b)$ holds, so $R_\phi$ does not define a function. Indeed, if we fix a sequence $(a_n)_{n\in\omega}$, the type $\{R_\phi(a_1,\ldots,a_m,y)\wedge R_\phi(a_1,\ldots,a_m,z)\wedge y\neq z \ | \ \phi\in E\}$ is consistent.
\end{example}

When we first learn about non-standard models of arithmetic, we might find the existence of infinite elements and its relative consistency with the axioms of arithmetic puzzling. Na\"{i}vely, we might try to axiomatise away these non-standard models. Perhaps we expand the structure by a predicate $U$, whose interpretation is all the standard integers, and extend our theory by the sentence saying that all elements are $U$. As we know, this does not work: non-standard models always exist. Passing from $\U$ to $\U^{heq}$ is exactly doing this. In hindsight, it is no surprise that $\textup{Th}(\U^{heq})$ does not axiomatise that the type $P$ from Remark~\ref{typedef quotient heq} defines a function. This is related to the fact that $\Pro(\Pro(\C))$ is in general not equivalent to $\Pro(\C)$. Likewise, $(\U^{heq})^{heq}$ is not the same as $\U^{heq}$ in any sense. 

If we wish to build a syntactic category of definable sets and functions from the structure $\U^{heq}$, we would want types like $P$ above to be morphisms. We have seen that $P$ is not a morphism in $\Pro(\defT(\textup{Th}(\U^{heq})))$. The problem lies in the fact that when we take the theory $\textup{Th}(\U^{heq})$, we lose information which is type-definable, for example that $P$ defines a function in $\U^{heq}$. Hence, when we extract a theory from $\U^{heq}$, we should instead do so in an infinitary positive logic. Consider that $\mathcal{U}^{heq}$ essentially is the structure of hyperdefinable sets of $\mathcal{U}$, where a set is hyperdefinable if it is the quotient of some type-definable equivalence relation. These hyperdefinable sets are closed under infinite conjunctions, infinite existential quantification and finite disjunction, but \emph{not} under negation. This corresponds exactly to infinitary positive logic, or the internal logic of a $\infty$-coherent category. We pursue this in the next subsection.

\medskip

\subsection{The infinitary syntactic category of a structure}
We study an infinitary positive variant of $\defT(T)$. We first define this semantically on a model.

\begin{defin}
    We define the \emph{infinitary syntactic category} $\tpdefT(\M)$ of a structure $\M$ as follows:
    \begin{itemize}
        \item Objects are type-definable sets of $\M$;
        \item Morphisms are type-definable sets of $\M$, which are the graphs of some function. 
    \end{itemize}
\end{defin}

Two different models of the same theory can give rise to different infinitary syntactic categories, but if the two models are sufficiently saturated, then the infinitary syntactic categories are equivalent. 

\begin{prop}
    Let $\M$ and $\N$ be monster models of the same theory $T$. Let $\tpdefT(\M)$ and $\tpdefT(\N)$ denote their infinitary syntactic categories. Then $\tpdefT(\M) \simeq \tpdefT(\N)$. 
\end{prop}

\begin{proof}
    The required functor sends the realisation of a type $X$ in $\M$ to the realisation of $X$ in $\N$. To ensure this is well defined, we prove that if $X$ and $Y$ are types with the same realisations in $\M$, then they have the same realisations in $\N$.

    Suppose not, say there is some $a\in\N$ realising $X$ but not $Y$. Then, there is some formula $\psi(x)$ satisfied by $a$, such that $q\vdash\neg\psi$. Then, the type $p\cup\{\psi\}$ is consistent and realised by some $b$ in $\M$. Then, $b$ realises $X$ but not $Y$ in $\M$.

    Next, we prove that if $P$ is a type-definable function from the type $X$ to the type $Y$ in $\M$, then this is also the case in $\N$. Suppose this is not the case in $\N$.

    \textit{Case} 1:\quad In $\N$, the set defined by $P$ is not total. Consider that the types $X$ and $\exists y \ (Y(y)\wedge X(x)\wedge P(x,y))$ are equivalent in $\M$. By above, they are also equivalent in $\N$. Hence, every $a\in X(\N)$ has an image in $Y$ under $P$.

    \textit{Case} 2:\quad In $\N$, the set defined by $P$ is not single-valued. Let $a\in X(\N)$ be such that there are distinct $b,b'\in Y(\N)$ and $P(a,b)\wedge P(a,b)$ holds. Let $y_0$ be the coordinate where $b,b'$ differ. Consider that the type $X(x)\wedge Y(y)\wedge Y(y')\wedge y_0\neq y'_0\wedge P(x,y)\wedge P(x,y')$ is consistent and realised in $\M$. Hence, $P$ is not single-valued in $\M$ either.
\end{proof}

Given that $\U$ is sufficiently saturated, the two categories $\tpdefT(\mathcal{U})$ and $\Pro(\defT(\textup{Th}(\mathcal{U})))$ are equivalent. This is an easy corollary of Lemma~\ref{functions of Pro}.

\begin{example}\label{positive not work ex}
In the Boolean setting, we can always present a morphism of $\tpdefT(\mathcal{U})$ as a morphism of $\Pro(\defT(T))$ by using compactness of non-positive formulas. This is in contrast to the setting when $T$ is a theory in positive logic. Let $T$ be the empty positive theory in the empty language with a single sort $S$. It is known that the positively closed models of $T$ are precisely the singleton sets \cite[Remark~2.1.12]{kamsma2025positivelogicintroductionmodel}. Hence, the formula $x=x$ is interpreted as the graph of a function $\top\rightarrow S$. However, $T$ does not prove that it is a function.
\end{example}

\medskip

\subsection{Connection with the exact completion}
We return to the case of $\mathcal{U}^{heq}$. 

\begin{thm}
    The category $\tpdefT(\U^{heq})$ is $\infty$-exact. 
\end{thm}

\begin{proof}
    The proof that $\tpdefT(\U^{heq})$ is $\infty$-regular is routine.

    We saw in Remark \ref{typedef quotient heq} that if an equivalence relation is type-definable in $T$, then it has a type-definable quotient in $\tpdefT(\U^{heq})$. 

    Now let $R([x_r]_E,[y_r]_E)$ be a type-definable equivalence relation in $\mathcal{U}^{heq}$, where $E$ is some type-definable equivalence relation in $T$. By Lemma \ref{heq translate}, there is an equivalence relation $\Gamma(x_r,y_r)$ type-definable in $T$, such that for all $a_r,b_r\in\mathcal{U}$,
    \[ \mathcal{U}\models\Gamma(a_r,b_r) \ \Leftrightarrow \ \mathcal{U}^{heq}\models R([a_r]_E,[b_r]_E). \]

    Now $\Gamma$ has a corresponding sort $S_\Gamma$ in $\mathcal{U}^{heq}$. Moreover, this comes with a type-definable quotient map $g: \mathcal{U}^I\rightarrow S_\Gamma$, where $I = |x_r|$. Similarly, there is a type-definable quotient map $f:\mathcal{U}^I\rightarrow S_E$. We define the quotient map $h:S_E \rightarrow S_\Gamma$ by
    \[ h([x_r]_E) = [y_r]_\Gamma \ \Leftrightarrow \ \exists z_r \ (f(z_r)=[x_r]_E \wedge g(z_r) = [y_r]_\Gamma). \]
    Type-definable sets are closed under infinite conjunection and existential quantification, so $h$ is type-definable indeed.
\end{proof}

Since every $T$-definable set is type-definable in $\U^{heq}$, there is a natural inclusion $\defT(T) \hookrightarrow \tpdefT(\U^{heq})$. Moreover, by compactness, this inclusion is full. By the universal property of the pro-completion and the exact completion, we obtain the following induced $\infty$-regular functors.

\[\begin{tikzcd}
	{\defT(T)} & {\Pro(\defT(T))} & {(\Pro(\defT(T)))_{ex/reg}} \\
	&& \tpdefT(\U^{heq})
	\arrow[hook, from=1-1, to=1-2]
	\arrow[curve={height=12pt}, from=1-1, to=2-3]
	\arrow[hook, from=1-2, to=1-3]
	\arrow[dashed, from=1-2, to=2-3]
	\arrow["F", dashed, from=1-3, to=2-3]
\end{tikzcd}\]

For notation, we consider the functors $\defT(T)\hookrightarrow\Pro(\defT(T))\hookrightarrow(\Pro(\defT(T)))_{ex/reg}$ as inclusions. One can think of the functor $F$ as follows. An object of $\defT(T)$, which is an $\Lang$-formula, is sent to its realisations in $\mathcal{U}^{heq}$. An object of $\Pro(\defT(T))$, which is a type of $T$, is sent to the cofiltered limit of the realisations of its formulas. From \cite[Corollary 10]{KAMENSKY2007180}, we know this is precisely the realisations of the type itself. An object of $(\Pro(\defT(T)))_{ex/reg}$, which is a quotient of a type by a type-definable equivalence relation in $T$, is sent to its corresponding quotient in $\tpdefT(\U^{heq})$. 

\begin{thm}
    The functor $F:(\Pro(\defT(T)))_{ex/reg} \rightarrow \tpdefT(\U^{heq})$ is an equivalence of categories.
\end{thm}

\begin{proof}
    We prove that $F$ is essentially surjective and fully faithful.
    
    Essential surjectivity: \quad Let $\Gamma([x]_E)$ be a type in $T^{heq}$, where $E$ is an equivalence relation on $\mathcal{U}^I$ type-definable in $T$. Then there is a type $\Sigma(x)$ in $T$ such that for all $a\in\mathcal{U}$
    \[ \Gamma([a]_E) \ \Leftrightarrow \ \Sigma(a). \]
    Define $E\upharpoonright_\Sigma$ to be the equivalence relation which is $E$ restricted to $\Sigma$. Categorically, this is the pullback $E \times_{\mathcal{U}^I\times\mathcal{U}^I} (\Sigma\times\Sigma)$. It is not hard to see that the following is a quotient diagram. 
    \[F(E\upharpoonright_\Sigma) \rightrightarrows F(\Sigma) \rightarrow \Gamma\]
    In $(\Pro(\defT(T)))_{ex/reg}$, there is $\textup{coeq}(E\upharpoonright_\Sigma \rightrightarrows \Sigma)$. Now, $F$ preserves quotients, so we have
    \[F(\textup{coeq}(E\upharpoonright_\Sigma \rightrightarrows \Sigma)) \cong \Gamma.\]

\smallskip

    Fullness: \quad Let $f\in\Hom_{\tpdefT(\U^{heq})}(F(X/D),F(Y/E))$, where $X/D,Y/E$ denote the quotients of $X,Y$ by equivalence relations $D,E$ respectively, and $X,Y,D,E\in\Pro(\defT(T))$. In other words, $f$ is type-definable in $\mathcal{U}^{heq}$. By Lemma \ref{heq translate}, there is a relation $R$ from $X$ to $Y$ which is type-definable in $T$ and such that for all $a,b\in\mathcal{U}$,
    \begin{equation}
        R(a,b) \ \Leftrightarrow \ f([a]_D)=[b]_E. \tag{$\dagger$}
    \end{equation}
    Now, we can consider $R$ as an object in $\Pro(\defT(T))$ and check that it satisfies Definition \ref{ex/reg carboni}. 
    
    First, we show that $R = RD$. Note that $RD(x,y)$ is the type $\exists x' \ D(x,x') \wedge R(x',y)$. By Proposition \ref{isomorphism is calculated in monster}, it is sufficient to show that $RD$ and $R$ have the same realisations in $\mathcal{U}$. That $R(\mathcal{U})\subseteq RD(\mathcal{U})$ follows from the reflexivity of $R$. Conversely, if $D(a,a')\wedge R(a',b)$ holds for some $a,a',b\in\mathcal{U}$, then $f([a]_D) = f([a']_D) = [b]_E$. Thus, we have $R(a,b)$. Similarly, one may show that $R = ER$.

    Second, we show that $D\leq R^\circ R$. Now, $R^\circ R(x,x')$ is the type $\exists y \ (R(x,y)\wedge R(x',y))$. It suffices to show that there is a monomorphism $D(x,x') \hookrightarrow R^\circ R(x,x')$ in $\Pro(\defT(T))$ given $(x,x')\mapsto (x,x')$. Now, this is a type-definable set, whose realisations form the graph of a function from $D(\mathcal{U})$ to $R^\circ R(\mathcal{U})$. By Lemma \ref{functions of Pro}, this corresponds to a morphism $\iota\in\Hom_{\Pro(\defT(T))}(D,R^\circ R)$. 
    To show that $\iota$ is a monomorphism, we consider $\Pro(\defT(T))$ as a subcategory of \textbf{Set} by interpreting in $\mathcal{U}$. Certainly, $\iota(\mathcal{U})$ is injective, so is a monomorphism in \textbf{Set}. Hence, it is also a monomorphism in $\Pro(\defT(T))$. Similarly, one may show that $RR^\circ \leq E$.
    
    Hence, $R\in\Hom_{(\Pro(\defT(T)))_{ex/reg}}(X/D,Y/E)$. For notation, we write $\tilde{f}$ for this morphism and reserve $R$ for the object.
    We claim that $F(\tilde{f})=f$. The condition $(\dagger)$ means that in ${\tpdefT(\U^{heq})}$, the following is a pullback square.
    \[\begin{tikzcd}
    	{F(R)} & {F(X\times Y)} \\
    	{\textup{graph}(f)} & {F(X/D\times Y/E)}
    	\arrow[hook, from=1-1, to=1-2]
    	\arrow[two heads, from=1-1, to=2-1]
    	\arrow["\lrcorner"{anchor=center, pos=0.125}, draw=none, from=1-1, to=2-2]
    	\arrow[two heads, from=1-2, to=2-2]
    	\arrow[hook, from=2-1, to=2-2]
    \end{tikzcd}\]
    In $(\Pro(\defT(T)))_{ex/reg}$, the morphism $\tilde{f}$ being the morphism induced by $R$ means that we have the following diagram.
    \[\begin{tikzcd}
        R & {X\times Y} \\
        {\textup{graph}(\tilde{f})} & {X/D\times Y/E}
        \arrow[hook, from=1-1, to=1-2]
        \arrow[two heads, from=1-1, to=2-1]
        \arrow["\lrcorner"{anchor=center, pos=0.125}, draw=none, from=1-1, to=2-2]
        \arrow[two heads, from=1-2, to=2-2]
        \arrow[hook, from=2-1, to=2-2]
    \end{tikzcd}\]
    Hence, we have 
    \[\textup{graph}(F(\tilde{f})) \cong F(\textup{graph}(\tilde{f})) \cong \textup{graph}(f), \]
    where the first isomorphism comes from the regularity of the functor $F$, and the second isomorphism comes from the uniqueness of the regular-epimorphism-monomorphism factorisation in ${\tpdefT(\U^{heq})}$. Then, $F(\tilde{f}) = f$ by the universality of taking the graph of a morphism.

\smallskip

    Faithfulness: \quad Let $f,g\in\Hom_{(\Pro(\defT(T)))_{ex/reg}}(X/D,Y/E)$, so there are relations $R,S$ from $X$ to $Y$ in $\Pro(\defT(T))$ such that the following two squares are pullback squares.
    \[\begin{tikzcd}
        R & {X\times Y} && S & {X\times Y} \\
        {\textup{graph}(f)} & {X/D\times Y/E} && {\textup{graph}(g)} & {X/D\times Y/E}
        \arrow[hook, from=1-1, to=1-2]
        \arrow[two heads, from=1-1, to=2-1]
        \arrow["\lrcorner"{anchor=center, pos=0.125}, draw=none, from=1-1, to=2-2]
        \arrow[two heads, from=1-2, to=2-2]
        \arrow[hook, from=1-4, to=1-5]
        \arrow[two heads, from=1-4, to=2-4]
        \arrow["\lrcorner"{anchor=center, pos=0.125}, draw=none, from=1-4, to=2-5]
        \arrow[two heads, from=1-5, to=2-5]
        \arrow[hook, from=2-1, to=2-2]
        \arrow[hook, from=2-4, to=2-5]
    \end{tikzcd}\]
    Suppose $F(f)=F(g)$. Then in $\mathcal{U}$, for any $a,a'$ realising $X$ and any $b,b'$ realising $Y$, such that $R(a,b)$ and $S(a',b')$, if $D(a,a')$, then $E(b,b')$. Now, the relations $R,S$ satisfy Definition \ref{ex/reg carboni}. In particular, they are invariant under both equivalence relations $D,E$, and are total and single-valued over their quotients. 
    
    Suppose $R(a,b)$. As $S$ is total, $S(a,b')$ holds for some $b'$. By the above, $E(b,b')$ holds. $S$ is invariant under $E$, so $S(a,b)$ holds. The converse also holds by symmetry. Hence, $R=S$, so $f=g$.
\end{proof}

\begin{remark}
    In the positive setting, the functor $F$ still exists and is essentially surjective and faithful. However, as Example~\ref{positive not work ex}, it might not be full. The use of Lemma~\ref{functions of Pro} and Proposition~\ref{isomorphism is calculated in monster} in the proof for fullness prevents the proof from transferring to the positive setting.
\end{remark}

\medskip

\subsection{A negative result}
One might hope that there is some free construction on $T$, giving us a first-order (or even positive) theory that eliminates hyperimaginaries, similar to how $T^{eq}$ is the free first-order theory on $T$ that eliminates imaginaries. Categorically, the hope is that there is some subcategory $\C$ of $(\Pro(\defT(T)))_{ex/reg}$, such that $\Pro(\C) \simeq (\Pro(\defT(T)))_{ex/reg}$. If such a subcategory exists, then we can think of its internal theory as the free construction on $T$ that eliminates hyperimaginaries. Alas, this is not true. We show such a subcategory does not exist for Example~\ref{finset}, that is, there is no category, whose pro-completion is \textbf{CHaus}. Now, the dual of this statement---$\textbf{CHaus}^{op}$ is not the ind-completion of any category---is well-known. For example, it can be deduced from \cite{LIEBERMAN2023107245}, where they showed that $\textbf{CHaus}^{op}$ is not finitely concrete. We give a more direct proof here.

By Remark~\ref{cocompact}, if such a subcategory of \textbf{CHaus} exists, then it has to be the full subcategory consisting of the cocompact objects. We have the following characterisation of cocompact objects in \textbf{CHaus}. This is folklore, but it seems that the proof is not written down anywhere.

\begin{prop}
    The cocompact objects of \textup{\textbf{CHaus}} are precisely the finite discrete spaces.
\end{prop}

\begin{proof}
    $(\Rightarrow)$ \quad Suppose $X\in\textbf{CHaus}$ is cocompact. Let $X_d$ be the space with the same underlying set as $X$ and the discrete topology. Let $\beta X_d$ be the Stone-\v{C}ech compactification of $X_d$. Then, there exists a unique continuous function $f:\beta X_d\rightarrow X$ extending the obvious morphism $\iota:X_d\rightarrow X$. Now, $\beta X_d$ is a Stone space, so it is a cofiltered limit of finite discrete spaces, say $\beta X_d\cong \lim F_i$, where each $F_i$ is finite. Since $X$ is cocompact, $f$ factors through some $F_i$. Hence, we obtain the following diagram.
    \[\begin{tikzcd}
    	{X_d} & X \\
    	{\beta X_d} & {F_i}
    	\arrow["\iota", from=1-1, to=1-2]
    	\arrow[from=1-1, to=2-1]
    	\arrow["f"{description}, from=2-1, to=1-2]
    	\arrow["{\pi_i}"', from=2-1, to=2-2]
    	\arrow["{f_i}"', from=2-2, to=1-2]
    \end{tikzcd}\]
    Since $\iota$ is a surjection, $f_i$ is also surjective. Since $F_i$ is finite, $X$ also has to be finite. 

    \smallskip

    $(\Leftarrow)$ \quad Let $F$ be a finite set with the discrete topology. Let $(X_i)_{i\in I}$ be a cofiltered system in \textbf{CHaus}, let $X$ denote the limit, let $p_i:X\rightarrow X_i$ denote the projection, and let $f:X \rightarrow F$ be a continuous function. Since $F$ is finite discrete, we can write $X$ as the disjoint union of the fibres $f^{-1}(a)$ for $a\in F$. Each fibre is clopen.
    
    Now, $X$ has the limit topology, so has basis given by $\{p_i^{-1}(V) \ | \ i\in I, V\subseteq X_i \textup{ open}\}$. For each $x\in X$, we choose some basic open $p_{i_x}^{-1}(V_x)\ni x$, such that $p_{i_x}^{-1}(V_x) \subseteq f^{-1}(f(x))$. This forms a cover of $X$. Since $X$ is compact, there is a finite subcover $\{p_{i_1}^{-1}(V_1),\ldots,p_{i_n}^{-1}(V_n)\}$.
    
    Since $I$ is cofiltered, there exists $i\in I$ with transition maps $p_{ii_r}:X_i\rightarrow X_{i_r}$ for $r=1,\ldots,n$, such that $p_{ii_r}p_i=p_{i_r}$. For each $r$, let $W_r = p_{ii_r}^{-1}(V_r) \subseteq X_i$. Then $p_i^{-1}(W_r) = p_{i_r}^{-1}(V_r)$, so $\{p_i^{-1}(W_r)\}_{r=1}^n$ is a cover of $X$. Moreover, since we have chosen each basic open to be contained in a single fibre of $f$, there exists a unique continuous map $h:\bigcup_{r=1}^n W_r \rightarrow F$ extending $f$ along $p_i$. 

    Note that $p_i(X)\subseteq W:= \bigcup W_r \subseteq X_i$. We show that there is some $j$ with a transition map $p_{ji}:X_j\rightarrow X_i$, whose image is contained in $W$. Otherwise, for all $j\in I$ with a morphism $j\rightarrow i$, the closed sets $p_{ji}(X_j)-W$ are non-empty and form a cofiltered system of compact Hausdorff spaces. Therefore, the limit, which is a subset of $X$, is non-empty, contradicting the fact that $p_i(X)\subseteq W$. Hence, fix $j$ such that $p_{ji}(X_j)\subseteq W$ and write $p_{ji}^W:X_j\rightarrow W$ for the same map with codomain restricted to $W$. Then, $f$ factors through $X_j$ via the map $h p^W_{ji}:X_j \rightarrow F$.
\[\begin{tikzcd}
	X && F \\
	{X_j} \\
	{X_i} & W & {W_r} \\
	{X_{i_r}} && {V_r}
	\arrow["f", from=1-1, to=1-3]
	\arrow["{p_j}"', from=1-1, to=2-1]
	\arrow["{p_{ji}}"', from=2-1, to=3-1]
	\arrow["{p_{ji}^W}", from=2-1, to=3-2]
	\arrow["{p_{ii_r}}"', from=3-1, to=4-1]
	\arrow["h"', from=3-2, to=1-3]
	\arrow[hook, from=3-2, to=3-1]
	\arrow[hook, from=3-3, to=3-2]
	\arrow["\lrcorner"{anchor=center, pos=0.125, rotate=-90}, draw=none, from=3-3, to=4-1]
	\arrow[from=3-3, to=4-3]
	\arrow[hook, from=4-3, to=4-1]
\end{tikzcd}\]
\end{proof}

The proposition shows that we have $\textup{Cocompact}(\textbf{CHaus}) \simeq \textbf{FinSet}$, whose pro-completion is \textbf{Stone}. In particular, there is no subcategory of \textbf{CHaus}, whose pro-completion is \textbf{CHaus} itself. 

\medskip

We conclude by comparing a few categories obtained by applying the pro-completion and the exact completion to $\defT(T)$ in different ways. The following diagram is obtained by inducing the dashed arrows in the labelled order using the universal properties of the pro-completion and the exact completion.
\[\begin{tikzcd}
	& {\defT(T)} && {\defT(T)_{ex/reg}} & {\defT(T^{eq})} \\
	& {\Pro(\defT(T))} && {\Pro (\defT(T)_{ex/reg})} \\
	{\tpdefT(\U^{heq})} & {(\Pro(\defT(T)))_{ex/reg}} && {(\Pro (\defT(T)_{ex/reg}))_{ex/reg}}
	\arrow[from=1-2, to=1-4]
	\arrow[from=1-2, to=2-2]
	\arrow["\simeq"{description}, draw=none, from=1-4, to=1-5]
	\arrow[from=1-4, to=2-4]
	\arrow["3"{description, pos=0.2}, dashed, from=1-4, to=3-2]
	\arrow["1"{description, pos=0.3}, dashed, from=2-2, to=2-4]
	\arrow[from=2-2, to=3-2]
	\arrow["4"{description}, dashed, from=2-4, to=3-2]
	\arrow[from=2-4, to=3-4]
	\arrow["\simeq"{description}, draw=none, from=3-1, to=3-2]
	\arrow["2"{description}, shift left=3, dashed, from=3-2, to=3-4]
	\arrow["\sim"{description}, draw=none, from=3-2, to=3-4]
	\arrow["5"{description}, shift left=3, dashed, from=3-4, to=3-2]
\end{tikzcd}\]

The bottom row is always an equivalence. If $T$ eliminates hyperimaginaries, then the bottom right three categories are equivalent. If $T$ eliminates imaginaries, then the three rows are equivalences. If $T$ eliminates both imaginaries and hyperimaginaries, then the bottom four categories are equivalent. Hence, we obtain our main theorem.

\begin{thm}
    \input{6main_theorem}
\end{thm}